\documentclass[11pt]{amsart}

\usepackage{amssymb,amscd}
 \usepackage{amsthm}

\begin{document}



\title{A Counterpart to Nagata Idealization}


\author{Bruce Olberding}

\address{Department of Mathematical Sciences, New Mexico State University,
Las Cruces, NM 88003-8001}

\begin{abstract}
Idealization of a module $K$ over a commutative ring $S$  produces a ring having $K$ as an ideal, all of whose elements are nilpotent.  We develop a method that under suitable field-theoretic conditions produces from an $S$-module $K$ and derivation $D:S\rightarrow K$ a subring $R$ of $S$ that behaves like the idealization of $K$ but is such that when $S$ is a domain, so  is $R$.  The ring $S$ is contained in the normalization of $R$ but is finite over $R$ only when $R = S$.  We determine conditions under which $R$ is Noetherian, Cohen-Macaulay, Gorenstein, a complete intersection or a hypersurface.  When $R$ is local, then its ${\bf m}$-adic completion is the idealization of the ${\bf m}$-adic completions of $S$ and $K$.         
\end{abstract}



\maketitle

\newtheorem{thm}{Theorem}[section]
\newtheorem{lem}[thm]{Lemma}
\newtheorem{prop}[thm]{Proposition}
\newtheorem{cor}[thm]{Corollary}
\newtheorem{exmp}[thm]{Example}
\newtheorem{rem}[thm]{Remark}
\newtheorem{ques}[thm]{Question}
\newtheorem{defn}[thm]{Definition}

\input amssym.def

\newcommand{\ilim}{\mathop{\varinjlim}\limits}

\def\Jac{\mbox{\rm Jac$\:$}}
\def\Nil{\mbox{\rm Nil$\:$}}
\def\wt{\widetilde}
\def\Q{\mathcal{Q}}
\def\O{\mathcal{O}}
\def\ff{\frak}
\def\Spec{\mbox{\rm Spec}}
\def\ZS{\mbox{\rm ZS}}
\def\wB{\mbox{\rm wB}}
\def\type{\mbox{ type}}
\def\Hom{\mbox{ Hom}}
\def\rank{\mbox{ rank}}
\def\Ext{\mbox{ Ext}}
\def\Ker{\mbox{ Ker }}
\def\Max{\mbox{\rm Max}}
\def\End{\mbox{\rm End}}
\def\ord{\mbox{\rm ord}}
\def\l{\langle\:}
\def\r{\:\rangle}
\def\Rad{\mbox{\rm Rad}}
\def\Zar{\mbox{\rm Zar}}
\def\Supp{\mbox{\rm Supp}}
\def\Rep{\mbox{\rm Rep}}
\def\cal{\mathcal}
\def\p{{\rm{p}}}

\def\gen{\mbox{\rm gen$\:$}}
\def\Jac{\mbox{\rm Jac$\:$}}
\def\Nil{\mbox{\rm Nil$\:$}}
\def\ord{\mbox{\rm ord}}
\def\wt{\widetilde}
\def\Q{\mathcal{Q}}
\def\ff{\frak}
\def\Spec{\mbox{\rm Spec}}
\def\ZS{\mbox{\rm ZS}}
\def\wB{\mbox{\rm wB}}
\def\type{\mbox{ type}}
\def\Hom{{\rm Hom}}
\def\depth{{\rm depth}}
\def\Gen{{\rm Gen}}
\def\rank{\mbox{\rm rank}}
\def\Ext{{\rm Ext}}
\def\embdim{{\mbox{\rm emb.dim } }}
\def\length{{\mbox{\rm length }}}
\def\Ker{\mbox{\rm Ker }}
\def\pd{{\rm pd}}
\def\gr{{\rm gr}}
\def\Max{\mbox{\rm Max}}
\def\End{\mbox{\rm End}}
\def\l{\langle\:}
\def\r{\:\rangle}
\def\Rad{\mbox{\rm Rad}}
\def\Zar{\mbox{\rm Zar}}
\def\Supp{\mbox{\rm Supp}}
\def\Rep{\mbox{\rm Rep}}
\def\cal{\mathcal}
\def\p{{\rm{p}}}
 



\section{Introduction}


All rings in this article are commutative and have an identity.  Let $A$ be a ring, and let $L$ be an $A$-module.  The {\it idealization}, or {\it trivialization}, of the module $L$ is the ring $A \star L$, which is defined as an abelian group  
 to be $A \oplus L$, and whose ring multiplication is given by $(a_1,\ell_1)\cdot(a_2,\ell_2) = (a_1a_2,a_1\ell_2 +a_2\ell_1)$ for all $a_1,a_2 \in A$, $\ell_1,\ell_2 \in L$.  In particular, $(0,\ell_1) \cdot (0, \ell_2) = (0,0)$, and hence the $A$-module $L$ is encoded into $A \star L$ as an ideal whose square is zero.  
 Nagata introduced idealization in \cite{Na} to deduce primary decomposition of Noetherian modules from the primary decomposition of ideals in Noetherian rings.  Among other positive uses of idealization are arguments for smoothness and the determination of when a Cohen-Macaulay ring admits a  canonical module (\cite[Section 25]{Ma} and \cite{Reiten}, respectively). But idealization also serves as a flexible source of examples in commutative ring theory, since it allows one to create a ring having an ideal which reflects the structure of a well-chosen module.  However, as is clear from the construction, idealization introduces nilpotent elements, and hence the construction does not produce domains.  
 
In this article we develop a construction of rings that behaves analytically like the idealization of a module.  The structure of these rings $R$ is determined entirely by an overring $S$ of $R$ and an $S$-module $K$, and when $S$ is a domain, so is $R$.  The   extension $R \subseteq S$ is integral, but  is finite only if $R = S$.     Moreover, for certain elements $r$ in the ring $R$, $R/rR$ is  isomorphic to $S/rS \star K/rK$, so that in sufficiently small neighborhoods, $R$ is  actually an idealization of $S$ and $K$.  There are enough such elements $r$ so that if $R$ and $S$ are quasilocal with finitely generated maximal ideals, then $\widehat{R}$ is isomorphic to $\widehat{S} \star \widehat{K}$, where $\widehat{R}$, $\widehat{S}$ and $\widehat{K}$ are completions  in relevant ${\bf m}$-adic topologies.  Thus when the ring $R$ produced by the construction is a local Noetherian ring, it is analytically ramified, with ramification given in a clear way.

The construction of the ring $R$ from $S$ and $K$ is as a ring of ``anti-derivatives'' for a special sort of derivation from a localization of $S$ to a corresponding localization of $K$:  

\begin{defn} {\em Let $S$ be a ring, let $K$ be an $S$-module, and let $C$ be a multiplicatively closed subset of nonzerodivisors of $S$ that are also nonzerodivisors on $K$.  Then a subring $R$ of $S$ is {\it twisted by $K$ along $C$} if there is a $C$-linear derivation $D:S_C \rightarrow K_C$ such that:
(a)  $R = S \cap D^{-1}(K)$,
(b) $D(S_C)$ generates $K_C$ as an $S_C$-module, and
(c)  $S \subseteq  \Ker D + cS$ for all $c \in C$.
We  say that {\it $D$ twists $R$ by $K$ along $C$.} }
\end{defn}
 
 Note that for any derivation $D$ from a ring into a module $L$, the preimage $D^{-1}(K)$, with $K$ a submodule of $L$, is a ring. In particular, $\Ker D$ is a ring.  
We use the term ``twisted'' to  draw a loose analogy with the  notion of the twist of a graded module or ring.  In our case, however, rather than twist, or shift, a grading by an index, we twist the ring $S$ by a module $K$ to create $R$.   This shifting is illustrated, to name just a few instances, by how $K$ shifts the Hilbert function of ideals of $R$ (see \cite{OlbSub}), and also how $K$ ramifies the completion of $R$.
Given $S$, $K$ and $C$, whether such a subring $R$ and derivation  $D$ exist is in general not easy to determine, and we address this in Section 3.  It is condition (c) that proves hard to satisfy.  By contrast, condition (b) can be arranged by choosing $K$ so that $K_C$ is the $S_C$-submodule of the target of the derivation generated by $D(S_C)$, while (a) can be satisfied simply by assigning $R$ to be $S \cap D^{-1}(K)$.

We are specifically interested in when this construction produces Noetherian rings, and a stronger, absolute version of the notion is useful for this: 

\begin{defn} {\em Let $S$ be a domain with quotient field $F$, and let $R$ be a subring of $S$.  Let $K$ be a torsion-free $S$-module, and let $FK$ denote the divisible hull $F \otimes_S K$ of $K$.  We say that
{\it $R$ is  strongly twisted by  $K$} if there is a derivation $D:F \rightarrow FK$ such that:
(a)  $R = S \cap D^{-1}(K)$,
(b) 
 $D(F)$ generates $FK$ as an $F$-vector space,
and
(c) $S \subseteq  \Ker D + sS$ for all $0\ne s \in S$.
We say that {\it $D$ strongly twists $R$ by $K$.}}
 \end{defn}
 
  It is straightforward to check that strongly twisted implies twisted along the multiplicatively closed set  
 $C=( S \cap \Ker D) \setminus \{0\}$.
 In Theorem~\ref{existence strongly twisted cor}, we prove that when $F$ is a field  of positive characteristic that is separably generated of infinite transcendence degree over a countable subfield, then for any domain $S$ having quotient field $F$ and torsion-free $S$-module $K$ of at most countable rank, there 
 exists a subring $R$ of $S$ that is strongly twisted by $K$.
Granted existence, we show in Theorem~\ref{Twisted Noetherian subrings} that if $R$ is a subring of a domain $S$ that is strongly twisted by a torsion-free $S$-module $K$, then $R$ is Noetherian if and only if $S$ is Noetherian and certain homomorphic images of $K$ are finitely generated.  In particular, if $S$ is a Noetherian ring and $K$ is a finitely generated torsion-free $S$-module, then $R$ is Noetherian.  Along with the above existence result, this guarantees there are interesting  examples which reflect in various ways  the natures of $S$ and $K$.    

The original idea of using pullbacks of derivations to construct Noetherian rings  is due to Ferrand and Raynaud, who used it to produce three important examples:  a one-dimensional local Noetherian domain $D$ whose
completion when tensored with the quotient field of $D$ is not a
Gorenstein ring  (in other words, its generic formal fiber  is not Gorenstein); a
 two-dimensional local Noetherian domain whose completion has embedded
 primes; and
a three-dimensional local Noetherian domain $R$ such that  the set
of prime ideals $P$ of $R$ with $R_P$ a Cohen-Macaulay ring  is not
an open subset of Spec$(R)$ \cite{FR}.  This last example was in fact constructed using the two-dimensional ring obtained in the second example, so   the construction is known only to produce examples in Krull dimension 1 and 2.  The method of Ferrand and Raynaud was further abstracted and improved on by 
Goodearl and Lenagan in the article \cite{GL}, but again, only examples of dimension 1 and 2 were produced.  
Our variation on these ideas  produces Noetherian rings  without restriction on dimension.    By developing the construction  generally without much concern for the Noetherian case, we  create more theoretical space for examples than the
arguments of Ferrand and Raynaud permit.  (For
example, the construction of Ferrand and Raynaud requires {\it a priori}
that the pullback $R$ of the derivation is Noetherian, a
condition that can be hard to verify and one which seems to be the main obstacle to producing more examples with their method.)  For more applications of some of these ideas to the case of dimension $1$, see \cite{OlbAR}.


In later sections we use some elementary facts about derivations to prove many of our results.  Let $S$ be a ring, and let $L$ be an $S$-module.  A mapping   $D:S \rightarrow L$ is a {\it  derivation} if for all $s,t \in S$,  $D(s+t) = D(s) + D(t)$ and 
 $D(st) = sD(t) + tD(s)$.  If also $A$ is a subset of $S$ with $D(A) = 0$, then $D$ is an {\it $A$-linear derivation.} 
  The main properties of derivations we need are collected in (1.3).

\smallskip

{\bf (1.3)}  {\it The module $\Omega_{S/A}$ of K\"ahler differentials}.  Let $S$ be a ring and let $A$ be a subring of $S$.  There exists an $S$-module
 $\Omega_{S/A}$   and an $A$-linear derivation $d_{S/A}:S \rightarrow \Omega_{S/A}$, such that for every
  derivation $D:S \rightarrow L$, there exists a unique $S$-module
  homomorphism $\alpha:\Omega_{S/A} \rightarrow L$ such that  $D =
  \alpha \circ d_{S/A}$; see for example, \cite[pp.~191-192]{Ma}.  
  The actual construction of $\Omega_{S/A}$ is not needed here, but we do use the fact that 
   the image of $d_{S/A}$ in
  $\Omega_{S/A}$ generates $\Omega_{S/A}$ as an $S$-module \cite[Remark 1.21]{Kunz}.
The $S$-module $\Omega_{S/A}$  is the {\it module of K\"ahler differentials}
 of the
 ring extension $A \subseteq S$, and the derivation  $d_{S/A}:S \rightarrow \Omega_{S/A}$ is the {\it  exterior differential} \index{exterior differential} of $A \subseteq S$.
 
\smallskip


\smallskip

We see in Theorem~\ref{pre-construction c}
that when $R$ is a twisted subring of $S$, then $R \subseteq S$ is a special sort of integral extension, which in \cite{OlbGFF} is termed  a ``quadratic'' extension:

\smallskip

{\bf (1.4)}  {\it Quadratic extensions.}
An extension $R \subseteq S$ is {\it quadratic} if every $R$-submodule of $S$ containing $R$ is a ring; equivalently, $st \in sR + tR + R$ for all $s,t \in S$.  In \cite[Lemma 3.2]{OlbGFF}, the following characterization is given for quadratic extensions in the sort of context we consider in this article.  Let 
 $R \subseteq S$ be an extension of rings,  and suppose there is  a multiplicatively closed subset $C$ of $R$ consisting of nonzerodivisors in $S$ such that every element of $S/R$ is annihilated by some element of $C$ and $S = R + cS$ for all $c \in C$.  (In the next section we will express these two properties by saying that $S/R$ is $C$-torsion and $C$-divisible.)    Then  $R \subseteq S$ is a quadratic extension if and only  there exists an
$S$-module $T$ and a  derivation $D:S \rightarrow T$ with  $R =
\Ker D$; if and only if $S/R$ admits an $S$-module structure extending the $R$-module structure on $S/R$.

\section{Analytic extensions}

In this section we introduce the notion of an analytic extension and show that twisted subrings are couched in such extensions, a fact that we rely heavily on  in later sections.  In framing the definition, and throughout this article, we use the following terminology.  
Let $S$ be a ring, let $L$ be an $S$-module and let $C$ be a multiplicatively closed subset of $S$ consisting of nonzerodivisors of $S$.  The module $L$ is {\it $C$-torsion} 
 provided that for each $\ell \in L$, there exists $c \in C$ with $c\ell =0$; it is {\it $C$-torsion-free} if the only $C$-torsion element is $0$.  The module $L$ is {\it $C$-divisible} if for each $c \in C$ and $\ell \in L$, there exists $\ell' \in L$ such that $\ell = c\ell'$.

  Following Weibel in \cite{Wei}, and as developed in \cite{OlbAR}, we use the following notion:
  
  \begin{defn} {\em Let  $\alpha:A \rightarrow S$ be a homomorphism of $A$-algebras, and let $C$ be a multiplicatively closed subset of $A$ such that the elements of $\alpha(C)$ are nonzerodivisors of $S$.  
 Then $\alpha$ is an {\it analytic isomorphism along $C$} if for each $c \in C$, the induced mapping $\alpha_c:A/cA \rightarrow S/cS:a \mapsto \alpha(a) + cS$ is an isomorphism.  When $A$ is a subring of $S$ and the mapping $\alpha$ is the inclusion mapping, we say that $A \subseteq S$ is a {\it $C$-analytic extension}.}
 \end{defn}

   It follows that the mapping $\alpha$ is analytic along $C$ if and only if $S/\alpha(A)$ is $C$-torsion-free and $C$-divisible.    For example, if $A$ is a ring and $X$ is an indeterminate for $A$, then $A[X] \subseteq A[[X]]$ is $C$-analytic with respect to $C = \{X^i:i>0\}$.  Similarly, if $A$ is a DVR with completion $\widehat{A}$, then $A[X] \subseteq \widehat{A}[X]$ is $C$-analytic for $C = \{t^i:i>0\}$, where $t$ is a generator of the maximal ideal of $A$.  
  
  We also consider a stronger condition:
  
  \begin{defn} {\em When $A$ and $S$ are domains,   the map $\alpha$ is a {\it strongly analytic isomorphism} if $sS \cap \alpha(A) \ne 0$ for all $0 \ne s \in S$ and $\alpha$ is an analytic isomorphism along $C = A \setminus \{0\}$.   When $A \subseteq S$ and the inclusion mapping is a strongly analytic isomorphism, then $A \subseteq S$ is a {\it strongly analytic extension}.}
  \end{defn}

 It follows that $A \subseteq S$ is a strongly analytic extension if and only if $S/A$ is a torsion-free divisible $A$-module and $P \cap A \ne 0$ for all nonzero prime ideals $P$ of $S$.  The latter condition asserts that the generic fiber of $\Spec(S) \rightarrow \Spec(A)$ is trivial.  Thus, following Heinzer, Rotthaus and Wiegand in \cite{HRW}, we say that the extension $A \subseteq S$ has {\it trivial generic fiber (TGF)}.    An immediate extension of DVRs is easily seen to be strongly analytic, but  examples of strongly analytic extensions of Noetherian rings in higher dimensions are harder to find.  One of the main goals of Section 3 is to give existence results for such extensions.

\begin{rem} \label{old remark} {\em It is straightforward to verify that an extension of rings $A \subseteq S$ is $C$-analytic, where $C$ is a multiplicatively closed subset of $A$ consisting of nonzerodivisors of  $S$, if and only if $S_C = A_C +S$ and $A = S \cap A_C$.  Moreover, if $A \subseteq S$ is an extension of domains with quotient fields $Q$ and $F$, respectively, then $A \subseteq S$ is strongly analytic if and only if $F = Q +S$ and $A = S \cap Q$.}  
\end{rem}


The following basic proposition  shows that for a $C$-analytic
extension $A \subseteq S$,  the ideals of $A$ and $S$ meeting $C$
are related in a  transparent way.

\begin{prop} \label{basic cor}  Let $A \subseteq S$ be an extension of rings, and let $C$ be a multiplicatively closed subset of $A$ consisting of nonzerodivisors of $S$. Suppose that $A \subseteq S$ is a  ${{\it C}}$-analytic extension.  Then:

\begin{itemize}

\index{Canalytic@$C$-analytic extension!and ideals}

\item[{(1)}]  The mappings $I \mapsto IS$ and $J \mapsto J \cap A$ yield a
one-to-one correspondence between ideals $I$ of $A$ meeting ${{\it
C}}$ and ideals $J$ of $S$ meeting ${{\it C}}$. Prime ideals of $A$
meeting ${{\it C}}$ correspond to prime ideals of $S$ meeting ${{\it
C}}$, and maximal ideals of $A$ meeting $C$ correspond to maximal
ideals of $S$ meeting $C$.

\item[{(2)}]  If $J$ is a finitely generated ideal of $S$ meeting ${{\it C}}$ that can be generated
by $n$ elements, then $J \cap A$ can be generated by $n+1$ elements.
If also $A$  is quasilocal, then $J \cap A$ can be
generated by $n$ elements.

\end{itemize}

\end{prop}

\begin{proof}  (1)  Let $I$ be an ideal of $A$ meeting ${{\it C}}$, and let $c \in I \cap S$.  Since $S = A+cS$ and $cS \cap A = cA$, it follows  that $I = IS \cap A$.
%
Similarly, if  $J$ is an ideal of $S$ meeting ${{\it C}}$ and $c \in J \cap C$, then from $S = A + cS$, 
 we deduce that
$J = (J \cap A)S$.
This proves that the mappings $I \mapsto IS$ and $J \mapsto J\cap A$
form a one-to-one correspondence.  The second assertion regarding
prime ideals is now clear, with one possible exception: If $P$ is a
prime ideal of $A$ meeting ${{\it C}}$, then since $S/A$ is
$C$-divisible, $S = A + PS$, so that $S/PS \cong A/(PS \cap A) =
A/P$. Hence $S/PS$ is a domain, and $PS$ is prime ideal of $S$
meeting ${{\it C}}$. This same argument shows also that if $P$ is a
maximal ideal of $A$, then $PS$ is a maximal ideal of $S$.  And
conversely, if $M$ is a maximal ideal of $S$ meeting $C$ and $P$ is  a prime
ideal of $A$ containing $M \cap A$, then the maximality of $M$
implies $M = (M \cap A)S = PS$, and hence by the correspondence, $M
\cap A = P$, so that maximal ideals of $S$ contract to maximal
ideals of $A$.

(2)  Suppose that  $J = (x_1,\ldots,x_n)S$ is a finitely generated
ideal of $S$ such that $c \in J \cap C$. By (1), $J = (J \cap A)S$,
so since $S = A + c^2S$, it follows that $J = (J \cap A)S= (J\cap A) + c^2S$, and hence  for each $i$, there exist $a_i \in J \cap
A$ and $s_i \in S$ such that $x_i = a_i + c^2s_i$. Thus $J =
(x_1,\ldots,x_n)S = (a_1,\ldots,a_n,c^2)S$, and by (1), $J \cap A =
(a_1,\ldots,a_n,c^2)A$.  To prove the last assertion, suppose that $A$ is quasilocal with maximal ideal ${\ff m}$,  and let $I = J \cap A$. Then since $c^2 \in {\ff m}I$, Nakayama's Lemma implies that $I = (a_1,\ldots,a_n)A$.
Hence $I$ can be generated by $n$ elements.
\end{proof}

Some relevant technical properties of analytic extensions were studied in  \cite{OlbAR}.  We quote these as needed throughout the article, beginning with the proof of the next theorem, which shows that twisted subrings occur within analytic extensions, and more interestingly, that a converse is also true.




\begin{thm} \label{C connection} Let $R \subseteq S$ be an extension of rings, and  let $C$ be a multiplicatively closed set of $R$ consisting of  nonzerodivisors of $S$.
\begin{itemize}

\item[{(1)}]  Suppose  $R$ is twisted along $C$ by an $S$-module $K$ that is $C$-torsion-free.  If $D$ is the derivation that twists $R$, then with $A = S \cap \Ker D$, the extension $A \subseteq S$ is  $C$-analytic, $R \subseteq S$ is quadratic and $S/R$ is $C$-torsion.

\item[{(2)}]  Conversely,  if there exists a subring $A$ of $R$ containing $C$ such that $A  \subseteq S$ is $C$-analytic, $R \subseteq S$ is quadratic, and $S/R$ is $C$-torsion, 
 then there is a unique $S$-submodule $K$ of $\Omega_{S_C/A_C}$ such that  $R$ is twisted  along $C$ by $K$, $\Omega_{S_C/A_C}/K$ is $C$-torsion and $d_{S_C/A_C}$ is the derivation that twists $R$.

\end{itemize}
\end{thm}

\begin{proof}
(1)    Since $D(S_C) \subseteq K_C$ and $D$ is an $A$-linear map, it follows that $S/A$ embeds in $K_C$ as an $A$-module.  Consequently, $S/A$ is $C$-torsion-free, since $K_C$ is $C$-torsion free.  Moreover, since $R$ is twisted by $D$, for each $c \in C$ we have $S \subseteq \Ker D + cS$, which in turn implies that $S = A + cS$.  Thus $S/A$ is $C$-divisible, and $A \subseteq S$ is a $C$-analytic extension.  Finally, we show that $S/R$ is a $C$-torsion
module and $R \subseteq S$ is a quadratic extension.  The former is the case, since if $s \in S$, then since $K_C/K$ is
$C$-torsion, there exists $c \in C$ such that $D(cs) = cD(s) \in K$,
and hence $cs \in S \cap D^{-1}(K)=R$.
  To see that $R \subseteq S$ is quadratic, we use (1.4).
Consider the derivation $$D':S \rightarrow K_C/K:s \mapsto D(s) + K,$$ where $s\in S$.  Then $\Ker D' = \{s \in S:D(s) \in K\} = S \cap D^{-1}(K) =R$, so that since $S/R$ is $C$-torsion (as we have verified) and $C$-divisible (it is the image of the $C$-divisible $A$-module $S/A$), we may apply (1.4) to conclude  that $R \subseteq S$ is a quadratic extension.

 (2) Write $d=d_{S_C/A_C}$ and $\Omega = \Omega_{S_C/A_C}$.  Define $K$ to be the $S$-submodule of $\Omega$ generated by $d(R)$.  
We claim  $R$ is twisted along $C$ by  $K$, and   
 the derivation that twists $R$ is  $d$.
Since $d$ is $C$-linear and $R_C = S_C$, we have:
  $$K_C = \sum_{r \in R}S_Cd(r) = \sum_{x \in R_C}S_Cd(x) = \Omega.$$
 Hence $K_C = \Omega$ and $K_C$ is generated as an $S_C$-module by $d(S)$.
Moreover,  $d:S_C \rightarrow K_C$ is a
$C$-linear derivation.  It is shown in \cite[Proposition 3.3]{OlbAR} that since $A \subseteq S$ is $C$-analytic, $R \subseteq S$ is quadratic and $S/R$ is $C$-torsion, then with $K$ defined as above, $R = S \cap  
 d^{-1}(K)$ and $\Omega/K$ is $C$-torsion, and  $K$ is the unique $S$-submodule of $\Omega$ satisfying these last two  properties.    
  Finally, since $A \subseteq S$ is $C$-analytic, $S = A + cS$ for all $c \in C$, and hence, since $A \subseteq \Ker d$, we have $S \subseteq \Ker d + cS$ for all $c \in C$.
Thus $R$ is twisted along $C$ by the $S$-module $K$.
\end{proof}

There is also a version of the  theorem for strongly analytic extensions.  Recall that if $S$ is a domain and $L$ is a torsion-free $S$-module, then a submodule $K$ of $L$ is {\it full} if $L/K$ is a torsion $S$-module.

\begin{cor} \label{connection}  Let $R \subseteq S$ be an extension of domains, and let $F$ denote the quotient field of $S$.
\begin{itemize}

\index{strongly twisted subring!as a strongly analytic sandwich}

\item[{(1)}]  Suppose  that $R$ is strongly twisted by a torsion-free $S$-module $K$.  If $D$ is the derivation that strongly twists $R$, then with $A = S \cap \Ker D$, the extension  $A  \subseteq S$ is strongly analytic, $R \subseteq S$ is 
quadratic and $R$ has quotient field $F$.

\index{analytic sandwich!as a strongly twisted subring}

\item[{(2)}]  Conversely,  if there exists a subring $A$ of $R$  such that $A  \subseteq S$ is strongly analytic, $R\subseteq S$ is quadratic and $R$ has quotient field $F$, then there is a unique 
 full $S$-submodule $K$ of $\Omega_{F/Q}$ such that  $R$ is strongly twisted   by $K$ and $d_{F/Q}$ is the derivation that twists $R$.

\end{itemize}
\end{cor}

\begin{proof}  (1)     First observe that every nonzero ideal of $S$ contracts to a nonzero ideal of $A$.  For if $s$ is a nonzero nonunit in $S$, then by assumption $S \subseteq \Ker D +s^2S$, so that $s = a + s^2 \sigma$ for some $a\in \Ker D \cap S = A$ and $\sigma \in S$.  Thus $s(1-s\sigma) \in A$, and since $s$ is a nonzero nonunit in $S$, it follows that $sS \cap A \ne 0$.  Therefore, 
 the extension $A \subseteq S$ has TGF.  Moreover,  $R$ is twisted by $K$ along $C:=A \setminus \{0\}$, so by Theorem~\ref{C connection}, $A \subseteq S$ is  $C$-analytic, $R \subseteq S$ is quadratic and $S/R$ is a torsion $R$-module.  Since $A \subseteq S$ has TGF and is analytic along $C = A \setminus \{0\}$, it follows that $A \subseteq S$ is strongly analytic.    

(2) Let $Q$ denote the quotient field of $A$, and let $C = A \setminus \{0\}$.  Since $A \subseteq S$ has TGF, it follows that $QS =S_C = F$.  Now by Theorem~\ref{C connection}, there exists a unique full $S$-submodule $K$ of
$\Omega_{F/Q}$ 
  such that $R$ is twisted by $K$ along $C$ by the  derivation
 $d_{F/Q}$.     Clearly, $d_{F/Q}$ generates $K_C=FK=\Omega_{F/Q}$ as an $F$-vector space.  Moreover, since $R$ is twisted by $D$ along $C$, we have 
$S \subseteq \Ker D + aS$ for all $0 \ne a \in A$.  Using again that $A \subseteq S$ has TGF, it follows that $S \subseteq \Ker D + sS$ for all $0 \ne s \in S$, and hence 
 $R$ is strongly twisted  by the $S$-module $K$.
\end{proof}

\section{Existence of strongly twisted subrings}

In this section we prove the existence of strongly twisted subrings of domains $S$ with sufficiently large quotient field $F$.  
When $S$ is, for example, a DVR, then this amounts to finding a subring $A$ of $S$ such that $A$ is a DVR and $A \subseteq S$ is an immediate extension, meaning that $A$ and $S$ have the same residue field and value group.  For given such a DVR $A$ with quotient field $Q$, then as in Corollary~\ref{connection}, a full $S$-submodule of $\Omega_{F/Q}$ gives rise to a strongly twisted subring of $S$; see \cite{OlbAR} for more on the special case of DVRs.  What makes the case of DVRs simpler is that an immediate  extension is easily shown to be strongly analytic.  However, in higher dimensions it is more of a challenge to find subrings $A$ of a given domain $S$ that induce a strongly analytic extension $A \subseteq S$ because such extensions must not only be analytic along $C = A \setminus \{0\}$, but must also have trivial generic fiber.   Satisfying these two conditions simultaneously is the obstacle.  

The first proposition of the section characterizes, but does not guarantee, the existence of strongly twisted subrings, and we do not use it again in this section when we prove existence results.  However, the proposition is a useful formulation for some classes of examples considered in Theorem~\ref{convex examples}. The proposition relies on a simple fact:        
Once a torsion-free module can be found that strongly twists a subring of $S$, then others also must exist, and it is really only the dimension of the $F$-vector space $FK$, i.e., the {\it rank} of $K$, \index{rank of a module} that is essential in guaranteeing the existence of other twisting modules.  This is the content of the following observation.

\begin{lem} \label{more} Let $S$ be domain, and suppose  $S$ has a subring that is strongly twisted by a  torsion-free $S$-module $K$. Then for every torsion-free $S$-module $L$ with $\rank(L) \leq \rank(K)$, there exists a subring of $S$ that is strongly twisted by $L$.
\end{lem}

\begin{proof}  Since $\rank(L) \leq \rank(K)$, there exists a projection of $F$-vector spaces $\alpha:FK \rightarrow FL$, so that $\alpha(FK) = FL$.  Let $D$ be the derivation that strongly  twists $K$, and let $D' = \alpha \circ D$.  Then $D'$ is a derivation taking values in $FL$, and since $D(F)$ generates $FK$ as an $F$-vector space, it follows that $D'(F) = \alpha(D(F))$ generates $FL$ as an $F$-vector space.  Moreover, $\Ker D \subseteq  \Ker \alpha \circ D = \Ker D'$, so that for all $0 \ne s \in S$, $S \subseteq \Ker D + sS \subseteq \Ker D' + sS$.  Therefore, $T:= S \cap D'^{-1}(L)$ is a subring of $S$ that is strongly twisted by $L$.
\end{proof}

\begin{prop} \label{very new} The following are equivalent for a domain $S$ with quotient field $F$.  

\begin{itemize}

\item[{(1)}] $S$ has a strongly twisted subring.

\item[{(2)}]  There exists  a nonzero derivation $D:F \rightarrow F$ such that  $S \subseteq \Ker D + sS$ for all $0 \ne s \in S$.

\item[{(3)}]  For each nonzero $S$-submodule $K$ of $F$, there exists a subring of $S$ that is strongly twisted by $K$.


\end{itemize}

\end{prop}

\begin{proof}
(1) $\Rightarrow$ (2)  Let $R$ be a subring of $S$ that is strongly twisted by a torsion-free $S$-module $K$, and let $D$ be the derivation that twists it.  Let  $\alpha:FK \rightarrow F$ be  an $F$-linear transformation such that $\alpha \circ D:F \rightarrow F$ is nonzero.  Then since $R$ is strongly twisted by $K$, we have  for all $s \in S$  that $S \subseteq \Ker D + sS \subseteq \Ker \alpha \circ D + sS$, so $\alpha \circ D$ is a derivation that behaves as in (2).   

(2) $\Rightarrow$ (3)  With $D$ as in (2), since $D$ is a nonzero deriviation and $F$ is a field, $D(F)$ generates $F$ as an $F$-vector space.  Thus  $S$ is a strongly twisted subring of itself (it is strongly twisted by $F$), so (3) follows from Proposition~\ref{more}.  

(3) $\Rightarrow$ (1) This is clear.  
%
%
 \end{proof}  

It follows that $S$ is a strongly twisted subring of itself if and only if $S$ satisfies the equivalent conditions of the proposition.

  The main focus of this section is  a  technique for producing strongly analytic extensions, and hence strongly twisted subrings.  
 Not surprisingly, we need
space to carve out such subrings, and hence we work with assumptions
involving infinite differential bases.  


\begin{lem} \label{pre dig} Let $F/k$ be an extension of
fields, where $k$ has at most countably many elements.

\begin{itemize}

\item[{(1)}]  When $F$ has characteristic $0$, then $|F| = \dim_F \Omega_{F/k}$  if and only if $F$ has infinite transcendence
degree over $k$.

\item[{(2)}]  If $F$ has positive characteristic, then in order that $|F| = \dim_F \Omega_{F/k}$, it suffices that  $F/k$ is separably generated and of infinite transcendence degree.
\index{separably generated extension}
\end{itemize}

\end{lem}

\begin{proof}
In both (1) and (2), there exists a transcendence basis
 $\{s_i:i \in I\}$ of $F$ over $k$ such that $F$ is separably algebraic over the subfield $k(s_i:i \in I)$.  Thus 
  $\{d_{F/k}(s_i):i \in I\}$ is a
basis for the $F$-vector space $\Omega_{F/k}$ \cite[Corollary A1.5(a), p.~567]{Eis}.  In particular,
 $|I| = \dim_F \Omega_{F/k}$.
  We claim that $|F| = |I|$.
  Since $F$ is algebraic over the infinite field
$F_0:=k(s_i:i \in I)$, then $|F|=|F_0|$ \cite[Lemma 2.12.6]{NewN}.
%
Also, since $k$ is countable and $I$ is infinite, the field  $F_0=k(s_i:i \in I)$ has cardinality $|I|$ (see for example the proof  of Theorem 2.12.5 in \cite{NewN}).
%
%
%
Therefore,
$|F|= |I| = \dim_{F} \Omega_{F/k}$.
\end{proof}


Our interest in such field extensions is that  in postive characteristic  they give rise to analytic
extensions:

\begin{lem} \label{dig} Let $ F/k $ be a 
extension of  fields with $|F| = \dim_F \Omega_{F/k}$, and suppose  $k$ has at most countably many elements;  $S$ is a
$k$-subalgebra of $F$ having quotient field $F$; and
 $L$ is an $F$-vector space of at most countable dimension.
Then for any $t \in S$ there exists a ring $A$ such that $k[t] \subseteq  A \subseteq S$ and
  $S/A$ is a torsion-free divisible
$A$-module.  For this ring $A$,  there exists an $A$-linear derivation $D:F \rightarrow L$ such
that $D(S) = L$.  \index{analytic extension!of characteristic $p$} If also $k$ has positive
characteristic, then $A \subseteq S$ is a strongly analytic extension.
\end{lem}

\begin{proof}  Since $S$ has quotient field $F$, the quotient rule for derivations
implies that the $F$-vector space
 $\Omega_{F/k}$ is
generated  by the set $\{d_{F/k}(s):s \in S\}$. Thus some subset of this collection is a basis for $\Omega_{F/k}$, and in fact if $d_{F/k}(t) \ne 0$, then we may choose this basis to contain the element $d_{F/k}(t)$.  (We allow throughout the proof the possiblity that $d_{F/k}(t) = 0$.)  Therefore, there is a  collection $\{s_i:i \in I\}$ of elements in $S$ such that the elements $d_{F/k}(s_i)$, $i \in I$, form a basis of $\Omega_{F/k}$, and if $d_{F/k}(t) \ne 0$, we can assume that $t \in \{s_i:i \in I\}$.

Viewing $F$ as a $k$-vector space, we claim that $\dim_k F = |I| = |F|$.
By assumption,
 $|F| = \dim_{F} \Omega_{F/k} = |I|$.  
 Clearly, $\dim_k F \leq |F| =  |I|$.  On the other hand, 
since $d_{F/k}$ is a $k$-linear map and  $\{d_{F/k}(s_i):i \in I\}$ is linearly independent, then  
  $\{s_i:i \in I\}$ is a  $k$-linearly independent subset of $F$.  
 Therefore,  $\dim_k F = |I| = |F|$.

Since $\dim_k F = |I|$,  we may  let $\{f_i:i \in I\}$ denote a basis
for $F$ over $k$.
We claim that $I$ can be  partitioned into countably many sets
$I_0, I_1,I_2,\ldots$ such that each $I_j$ has $|I|$ elements and if $d_{F/k}(t) \ne 0$, then $t = s_i$ for some $i \in I_0$.  Indeed,
first we may partition $I$ into a disjoint union $I =
\bigcup_{\alpha \in {\cal A}}X_\alpha$ of countably infinite sets
$X_\alpha$ \cite[Theorem 12, p.~40]{Kap}. For each $\alpha \in {\cal
A}$, write $X_\alpha = \{x_{\alpha,j}:j\in {\mathbb{N}} \cup \{0\}\}$.  Then for each $j \in {\mathbb{N}} \cup \{0\}$ we
define $I_j = \{x_{\alpha,j}:\alpha \in {\cal A}\}$.  Now $|I| =
|{\cal A}|\cdot \aleph_0,$ so since $|I|$ is infinite, $|I| = |{\cal
A}| $ \cite[Theorem 16, p.~40]{Kap}, and hence $|I| =  |I_j|$ for
all $j$.
 Thus for
each $j \in {\mathbb{N}} \cup \{0\}$, there exists a bijection $\sigma_j:I_j \rightarrow I$.  Moreover, if $d_{F/k}(t) \ne 0$, then after relabeling the $I_j$ we may assume that $t = s_i$ for some $i \in I_0$.  

 Let $L$ be an $F$-vector space of countably infinite dimension, and write $L =
 \bigoplus_{j=1}^\infty Fe_j$, where
$\{e_1,e_2,e_3,\ldots\}$ is a basis for $L$ over $F$.
 Define an $F$-linear mapping $\phi:\Omega_{F/k}
 \rightarrow L$ on our particular basis elements of $\Omega_{F/k}$ by
 $$\phi(d_{F/k}(s_i)) = \left\{ \begin{array}{ll}
 0 \; &  \mathrm{if}\, i \in I_0 \\
 f_{\sigma_j(i)}e_j \; & \mathrm{if}\,  i  \in I_j{\mbox{ with }} j \in {\mathbb{N}}
\end{array}\right.
 $$
 Recall that  if $d_{F/k}(t) \ne 0$, then $t = s_i$ for some $i \in I_0$.  Thus regardless of whether $d_{F/k}(t) \ne 0$, we have arranged it so  that $\phi(d_{F/k}(t)) = 0$.

 Now
 $D:=\phi\circ d_{F/k}:F \rightarrow L$ is a $k$-linear derivation with $D(t) = 0$.
 Moreover, for each $j \in {\mathbb{N}}$ (we are purposely excluding the case $j=0$ here), we have: $$D(\sum_{i \in I_j}k \cdot s_i) = \sum_{i \in I_j}k
 \cdot
 \phi(d_{F/k}(s_i)) = \sum_{i \in I_j} k \cdot f_{\sigma_j(i)}e_j =
 \sum_{i \in I}k \cdot f_ie_j = Fe_j.$$
 Thus: 
 $$L = \bigoplus_{j=1}^\infty Fe_j = \sum_{j=1}^\infty
D(\sum_{i \in I_j}k \cdot s_i) \subseteq  D( \sum_{i \in I}k \cdot s_i)
\subseteq D(S) \subseteq L,$$ which proves that $D(S) = L$.

Now set $A = S \cap \Ker D$.  Then $D$ is an $A$-module
homomorphism, and since $D(S) = L$, it follows that $S/A \cong L$
as $A$-modules. Since $F$ is a torsion-free divisible $A$-module and
$L$ is an $F$-vector space, then $L$, and hence $S/A$, are
torsion-free divisible $A$-modules.  Moreover, by the construction of $D$, we have $k[t] \subseteq  \Ker D \cap S = A$.
%
Furthermore, in the case where $k$ has characteristic
 $p \ne 0$, let $I$ be a nonzero ideal of $S$, and let $0 \ne s \in
 I$.  Then $D(s^p) = ps^{p-1}D(s) = 0$, so that $0 \ne s^p \in A \cap I$.
Thus $A \subseteq S$ has TGF, and $A \subseteq S$ is strongly analytic.  Finally, if $L'$ is a
finite dimensional $F$-vector space, then there is a surjective
$F$-linear transformation $\psi:L \rightarrow L'$, so that $\psi
\circ D:F \rightarrow L'$ is an $A$-linear derivation with $(\psi
\circ D)(S) = \psi(L) = L'$.
\end{proof}

The small detail in the  lemma allowing us to assume that $t \in A$ has an important consequence in that when $t$ is chosen  a nonunit in $S$, then since $t \in A$, there exist nonzero proper ideals in $S$ that contract to nonzero proper ideals in $A$.  Therefore, when $C = A \setminus \{0\}$, there exist proper ideals of $S$ meeting $C$, and in particular, $A$ is not a field.
  We use this observation in
  Corollary~\ref{existence V cor} 
   in a crucial way.

The following theorem is our main source of examples of strongly twisted subrings in high dimensions.  

\begin{thm}   \label{existence strongly twisted cor} Let $F/k$ be a field
 extension such that $k$ has positive characteristic and at most countably many elements.   Suppose that 
 $F/k$  is a separably generated extension of infinite transcendence degree.  If $S$ is a $k$-subalgebra of $F$ with quotient field $F$ and $K$ is a torsion-free $S$ module of at most countable rank, then there exists a subring $R$ of $S$ that is strongly twisted by $K$.
\end{thm}

\begin{proof}
Since $F$ is a separably generated extension of infinite transcendence degree over $k$, then by Lemma~\ref{pre dig}, $|F| = \dim_F \Omega_{F/k}$.  Let $L = FK$, the divisible hull of $K$.  Then $L$ is an $F$-vector space of at most countable dimension, so by Lemma~\ref{dig} there exists a $k$-subalgebra $A$ of $S$ such that $A \subseteq S$ is a strongly analytic extension and there exists an $A$-linear derivation $D:F \rightarrow L$ such that $D(S) = L$.  Since $A \subseteq S$ is strongly analytic, then $S/A$ is a divisible $A$-module, and consequently, $S = A + aS$ for all $0 \ne a \in A$.  If $0 \ne s \in S$, then since strongly analytic extensions have TGF, there exists $0 \ne a \in A \cap sS$, so that $S = A + sS$.  Moreover, $D$ is $A$-linear, so $A \subseteq \Ker D$, and hence $S \subseteq \Ker D + sS$ for all $0 \ne s \in S$.  Thus, setting $R = S \cap D^{-1}(K)$, we have that $R$ is strongly twisted by $K$.
\end{proof}

Thus 
 if $k$ is a field of positive characteristic  that is a separably generated extension of infinite transcendence degree over a countable subfield, and we choose $S$ between $k[X_1,\ldots,X_n]$ and $k(X_1,\ldots,X_n)$, then for every torsion-free $S$-module $K$ of at most countable rank, there exists a subring $R$ of $S$ that is strongly twisted by $K$.  Note however that the theorem does not assert  that  $R$ is a $k$-algebra.    

\section{Basic properties of twisted subrings}

We give now a few basic properties of twisted subrings, the most fundamental of which is the assertion in Theorem~\ref{pre-construction} that if $R$ is a subring of $S$ twisted by $K$, then  $R$ and $S\star K$ are isomorphic analytically, in the sense of Section 2.  Many of the results in this section depend on properties of analytic isomorphisms developed in  \cite{OlbAR}, which we refer to in the course of the proofs.  We assume throughout this section the following hypothesis.  

\begin{quote} {\it  $S$ is a ring; $C$ is a multiplicatively closed subset of nonzerodivisors of $S$;   
$R$ is a subring of $S$ that is twisted along $C$ by a $C$-torsion-free $S$-module $K$; and $D$ is the derivation that twists $R$.} \end{quote}

We note first that  $R \subseteq S$ is a quadratic, hence integral, extension.  In the case where $S$ is a domain and $R$ is strongly twisted, this extension is {\it subintegral} in the sense of Swan \cite{Swan}, meaning that 
 $R \subseteq S$ is 
integral, the contraction mapping $\Spec(S) \rightarrow \Spec(R)$ is
a bijection and the induced maps on residue field extensions are
isomorphisms (so for every prime ideal $P$ of
$S$, $S_P = R_{P \cap R} + PS_P$).

\begin{thm} \label{pre-construction c}  
  The rings $R$ and $S$ share the same total ring of quotients, and the  extension $R \subseteq S$ is quadratic, and hence integral, but not finite unless $R = S$. 
  If also $S$ is a domain, $K$ is torsion-free and $R$ is strongly twisted by $K$, then $R \subseteq S$ is a subintegral extension.
  \end{thm}

\begin{proof}
To see that $R$ and $S$ share the same total ring of quotients, it suffices to show that $R_C = S_C$.  Let $s \in S$.  Then since 
 $D(s) \in K_C$,  there exists $c \in C$ such that $cD(s) \in K$.  But $D$ is $C$-linear, so $D(cs) \in K$, and hence $cs \in D^{-1}(K) \cap S=R$, proving that $S_C = R_C$. 
By Theorem~\ref{C connection}(1), $R \subseteq S$ is a quadratic extension.  
If this extension is also finite, say, 
$S = Rs_1+\cdots+Rs_n$ for some $s_i \in S$, then  there exists $c \in C$ such that $cs_1,\ldots,cs_n \in R$, and hence $cS \subseteq R$.  But since $S/R$ is $C$-divisible, $S = R + cS$, so this forces $S = R$.

Finally, suppose that $S$ is a domain and $R$ is strongly twisted by $K$.  To see that $R \subseteq S$ is subintegral, we note first that by Proposition~\ref{basic cor}, the contraction mapping $\Spec(S) \rightarrow \Spec(R)$ is a bijection.  To complete the proof, let $P$ be a prime ideal of $S$.  We claim that $S_P = R_{P \cap R} + PS_P$. Since $S/R$ is a divisible $R$-module (this follows from the definition of strongly twisted), $S = R + P$, so it suffices to show that for each $b \in S \setminus P$, $b^{-1} \in R_{P \cap R} + PS_{P}$.   Let $b \in S \setminus P$, and let $A = S \cap \Ker D$. We claim that  $bS \cap A \not \subseteq P$.  For 
 by Corollary~\ref{connection}, $A \subseteq S$ is strongly analytic, so if $bS \cap A \subseteq P$, then by Proposition~\ref{basic cor}, $bS = (bS \cap A)S \subseteq P$, contrary to assumption.  Thus there exist $a \in A \setminus P$ and $s \in S$ such that $a = bs$.  Moreover, choosing $0 \ne c \in P \cap A$, we have since $S = A + acS$ that there exists $d \in A$ and $\sigma \in S$ such that $s = d + ac\sigma$, whence $b^{-1} = sa^{-1} = da^{-1} + c\sigma \in R_P + PS_P$, as claimed.  Therefore, $R \subseteq S$ is a subintegral extension.   
\end{proof}

As a consequence of the fact that $R \subseteq S$ is integral, along with the fact that $S/R$ is $C$-torsion and $C$-divisible, we obtain information about $\Spec(R)$:

\begin{thm} \label{prime correspond} \label{C-bijection}  The mappings $P \mapsto PS$ and $Q \mapsto Q \cap R$ define a
one-to-one correspondence between prime ideals $P$ of $R$ meeting
${{\it C}}$ and prime ideals $Q$ of $S$ meeting ${{\it C}}$. Under
this correspondence, maximal ideals of $R$ meeting $C$ correspond to
maximal ideals of $S$ meeting $C$.
 If also $S$ is a domain, then
the contraction mapping $\Spec(S) \rightarrow \Spec(R)$ is a bijection.
\end{thm}

\begin{proof} In the proof, we use only that $R \subseteq S$ is integral (a fact given by Theorem~\ref{pre-construction c}) and that   $S/R$ is $C$-divisible and $C$-torsion.
  Let $P$ be a prime ideal of $R$ meeting ${{\it C}}$.
Since $S/R$ is $C$-divisible, $S = R + PS$, and hence $S/PS \cong
R/(PS \cap R)$.  If $Q$ is any prime ideal of $S$ lying over $P$,
then $P \subseteq PS \cap R \subseteq Q \cap R = P$, so that $P = PS
\cap R$.  Hence $S/PS \cong R/P$, and it follows that $PS$ is a
prime ideal of $S$. This shows also that if in addition $P$ is a
maximal ideal of $R$, then $PS$ is a maximal ideal of $S$.  On the
other hand, if $L$ is a prime ideal of $S$ meeting $C$, then 
 since for $c \in L \cap C$,  $S  = R + cS$, it follows  that
 $L = (L \cap R)S$.
Also, since maximal ideals of an
integral extension contract to maximal ideals, it follows that maximal ideals of $R$ meeting $C$ correspond to
maximal ideals of $S$ meeting $C$.

Now suppose that $S$ is a domain. We show that the mapping $\Spec(S) \rightarrow \Spec(R)$ is a bijection.  
To this end,  we claim first that for each prime ideal $P$ of $R$ not meeting $C$, $R_P = S_{P_1}$, where $P_1$ is any prime ideal of $S$ lying over $P$, and hence  there is only one such prime ideal $P_1$.   (Note that since $S/R$ is $C$-torsion, it follows that $R$ and $S$ have the same quotient field so we may view all the localizations of $R$ and $S$ as subsets of this field.)
 Let $P$ be a prime ideal of $R$ not meeting $C$, and let $P_1$ be a prime ideal of $S$ lying over $P$.  Since $P \cap C$ is empty, $R_C \subseteq R_P$, and since $R_C = S_C$, we have $S \subseteq R_P$.  Thus $SR_P = R_P$, and it suffices to show that $SR_P = S_{P_1}$.  In fact, since $R_P \subseteq S_{P_1}$, we need only show that $S_{P_1} \subseteq SR_P$, or, more precisely, that each $b \in S \setminus P_1$ is a unit in $SR_P$.    Let $b \in S \setminus P_1$.  Since $SR_P = R_P$, $SR_P$ is a quasilocal ring with maximal ideal $PR_P$.  If $b \in PR_P$, then $b \in P_1S_{P_1} \cap S = P_1$, a contradiction.  Therefore, $b$ is in the quasilocal ring $SR_P$, but not in its maximal ideal, so $b$ is a unit in $SR_P$, and we have proved that $R_P = S_{P_1}$.

Now we claim that $\Spec(S) \rightarrow \Spec(R)$ is a bijection.  Since $R \subseteq S$ is integral, and hence
each prime ideal of $R$ has a prime ideal of $S$ lying  over it,
 the contraction mapping $\Spec(S) \rightarrow \Spec(R)$ is surjective.  To see that it is injective, let $P$ be a prime ideal of $R$,
  and let  $P_1$ be a  prime ideal of $S$ lying over $P$. If $P \cap C$ is empty, then by what we have established above, $R_P = S_{P_1}$, so that $P_1$ is the only prime ideal of $S$ lying over $P$.
    Otherwise, if $P \cap C$ is nonempty, then  by the first claim of the theorem, $P_1 = (P_1 \cap R)S = PS$, so that $P_1$ is the unique prime ideal of $S$ lying over $P$.  This completes the proof.
\end{proof}

The next theorem, which relies in a crucial way on \cite[Lemma 3.4]{OlbAR}, establishes a correspondence between submodules of $K_C/K$ and rings between $R$ and $S$.

\begin{thm} \label{correspondence}   There is a one-to-one correspondence between intermediate rings $R \subseteq T \subseteq S$ and $S$-submodules $L$ of $K_C$ with $K \subseteq L \subseteq K_C$ given by: $$T \mapsto \sum_{t \in T}SD(t) \:\:\:\:{\mbox{ and }} \:\:\:\: L \mapsto S \cap D^{-1}(L).$$
\end{thm}

\begin{proof}  Let $A = S \cap \Ker D$, and let $d = d_{S_C/A_C}$.  For each ring $T$ with $R \subseteq T \subseteq S$, define $\Omega(T) = \sum_{t \in T}Sd(t)$ and $L(T) = \sum_{t \in T}SD(t)$.
 By Theorem~\ref{C connection}, $A \subseteq S$ is a $C$-analytic extension. 
 Also, since $R$ is twisted by $K$ along $C$, then $R = S \cap D^{-1}(K)$ and $D(S_C)$ generates $K_C$ as an $S_C$-module.  
 In \cite[Lemma 3.4]{OlbAR}, it is shown that these two facts, along with the fact that $K$ is $C$-torsion-free, imply that  there exists a surjective $S_C$-module homomorphism $\alpha:\Omega_{S_C/A_C} \rightarrow K_C$ such that $D = \alpha \circ d$ and 
 for each $S$-module $L$ of $K_C$ with $L_C = K_C$, when $T = S \cap D^{-1}(L)$, then 
  $\alpha(\Omega(T)) = L$ and $\Omega(T)= \alpha^{-1}(L)$.

 To prove the theorem it suffices to show that for all 
  rings $T$ with $R \subseteq T \subseteq S$, we have 
  $T = S\cap D^{-1}(L(T))$, and for all
     $S$-modules $L$ with  $K \subseteq L \subseteq K_C$, we have $L = L(S \cap D^{-1}(L))$. 
Let $T$ be a ring between $R$ and $S$. 
Clearly, $T \subseteq S \cap D^{-1}(L(T))$.  Conversely, suppose that $s \in S \cap D^{-1}(L(T))$.  Then there exist $s_1,\ldots,s_n \in S$ and $t_1,\ldots,t_n \in T$ such that $D(s) = \sum_i s_i D(t_i)$.  
  Therefore, since $D = \alpha \circ d$, we have $\alpha(d(s)) = \alpha(\sum_i s_id(t_i))$.  Thus since  $\Ker \alpha \subseteq\alpha^{-1}(K) = \Omega(R)$, we have  $d(s) - \sum_i s_id(t_i) \in \Omega(R) \subseteq \Omega(T)$.  Therefore, $d(s) \in \Omega(T)$. As observed in the proof of Theorem~\ref{C connection}(2), $T = S \cap d^{-1}(\Omega(T))$, so  $s \in T$.  This proves that $T = S \cap D^{-1}(L(T))$. 
Finally, let $L$ be an $S$-module between $K$ and $K_C$.  We claim that $L = L(T)$, where $T = S \cap D^{-1}(L)$.  But this is immediate, since as noted above, $L = \alpha(\Omega(T)) = \sum_{t \in T}S\alpha(d(t)) = \sum_{t\in T}SD(t) = L(T)$.
  %
\end{proof}



Next we  show, in what is our main structure theorem for twisted subrings, that a twisted subring behaves analytically like an idealization.    The theorem is based on Proposition 3.5 in \cite{OlbAR}, and relies on the following lemmas. 

\begin{lem} \label{ontoness}
The derivation $D$   induces an isomorphism of $R$-modules given by $$\alpha:S/R \rightarrow K_C/K:s+R \mapsto D(s) + K.$$
\end{lem} 

\begin{proof}   Since $R = S \cap D^{-1}(K)$, it is clear that $\alpha$ is well-defined and injective.  To see that $\alpha$ is onto, let $y \in K_C$.  Then since $K_C$ is generated as an $S_C$-module by $D(S_C)$ and $D$ is $C$-linear, we may write $y = \sum_i {s_i}D(x_i)$, where  $s_i \in S$ and $x_i \in S_C$.  Choose $c \in C$ such that $cD(x_i) \in K$ for each $i$.  
Since $S \subseteq \Ker D  + cS$, we may for
each $i$ write $s_i = a_i + c\sigma_i$, where $a_i \in \Ker D$ and
$\sigma_i \in S$.  Thus, since $a_1,\ldots,a_n \in \Ker D$,  we have: $$y + K = \sum_{i}D(a_ix_i) +
\sum_{i}\sigma_i cD(x_i) +K =
D(\sum_i a_ix_i) + K.$$ Therefore, $D$ maps  onto $K_C/K$. 
\end{proof}

\begin{lem} \label{one case} Let $f:R \rightarrow S \star K$ be the ring homomorphism defined by $f(r) = (r,D(r))$ for all $r \in R$.  
 If $I$ is an ideal of $R$ meeting $C$ and $I = IS \cap R$, then $f(I)(S \star K)  
= (IS) \star K.$
\end{lem}

\begin{proof} 
 Clearly, $f(I)(S\star K) \subseteq (IS) \star K$.  To verify the reverse inclusion,
let $x
\in IS$, and let $k \in K$.  We show that $(x,0)$ and $(0,k)$ are
both in $f(I)(S \star K)$, since this is enough to prove the
claim. Let $A = S \cap \Ker D$.  Choose   $c \in I \cap {{\it C}}$. Then $S =A + cS$, so since $I = IS \cap R$,  it follows that  $x = i + cs$ for some $i \in I$ and $s \in S$.
Since $S/A$ is $C$-divisible, we may write $i = a + c\sigma$ for
some $a \in A$ and $\sigma \in S$.  Thus since we have assumed $I = IS \cap R$, we have $a = i - c\sigma \in IS \cap A
= (IS \cap R) \cap A = I \cap A$.  Consequently, setting $t =
\sigma + s$, we have $x = i + cs = a + c(\sigma + s) = a + ct$.
Since $D$ is $A$-linear, then $D(a) = 0$, and: $$(x,0) = (a,0) + (ct,0) = f(a) + f(c)(t,0)  \in f(I)(S\star K).$$ 
Next,  we show that $(0,k) \in f(I)(S \star K)$.
By Lemma~\ref{ontoness}, $K_C = D(S) + K$, so since $D$ is $C$-linear, there exist $s_2 \in S$ and $k_2 \in K$ such $k = D(cs_2) + ck_2$.  Since $R = S \cap D^{-1}(K)$, it follows that $y:=cs_2 \in IS \cap R = I$.  
Thus  $f(y) = (y,D(y))  \in f(I) (S
\star K)$. Consequently: $$(0,k) = (y,D(y)) + (-y,ck_2) = (y,D(y))+f(c)
 (-s_2,k_2) \in f(I) (S \star K).$$  This proves $f(I)
 (S \star K) = (IS) \star K$.
\end{proof}

\begin{thm} \label{pre-construction}  
The mapping $f:R \rightarrow S \star K:r \mapsto (r,D(r))$
 is an analytic isomorphism along $C$.  If also $S$ is a domain, $K$ is torsion-free and $R$  is strongly twisted by $K$, then this map is faithfully flat.  
\end{thm}

\begin{proof}
Let $A = S \cap \Ker D$.  Then by Theorem~\ref{C connection}, $A \subseteq S$ is $C$-analytic.  It is shown in \cite[Proposition 3.5]{OlbAR} that this fact along with the following assumptions imply that $f$ is an analytic isomorphism along $C$:  
(a)  $K$ is a $C$-torsion-free $S$-module;  (b) $D:S_C \rightarrow K_C$ is an $A_C$-linear derivation; (c) 
 $D(S_C)$ generates $K_C$ as an $S_C$-module, and (d)   $R  = S \cap D^{-1}(K)$.  All of these conditions are satisfied since $R$ is twisted by $K$ along $C$, and $D$ is the derivation that twists it.  If also $S$ is a domain, $K$ is torsion-free and $R$ is strongly twisted by $K$, then $f$ is an analytic isomorphism along $C = A \setminus \{0\}$, so that Coker $f$ is a torsion-free divisible $R$-module, a fact which implies that $f$ is flat.  
If $M$ is  a maximal ideal of $R$, then since by Theorem~\ref{pre-construction c}, $R \subseteq S$ is integral, it follows that $M = MS \cap R$.  Moreover, since $A \subseteq S$ is strongly analytic (Corollary~\ref{connection}), and hence has TGF, the extension $A \subseteq R$ has TGF since $S/R$ is a torsion $R$-module.  If $R$ is a field, then clearly $f$ is faithful.  Otherwise, if $R$ is not a field, 
then 
the maximal ideals of $S$ meet $C = A \setminus \{0\}$, and hence by Lemma~\ref{one case}, $f(M)(S\star K) = MS \star K \ne S \star K$, so that $f$ is faithful.   
\end{proof}

\begin{cor} \label{pre-construction completion}  
If $R$ and $S$ are quasilocal, each  
 with finitely generated maximal ideal meeting $C$, then
$f: R \rightarrow S \star K$ lifts to  an isomorphism of rings, $\widehat{R}
\rightarrow \widehat{S} \star \widehat{K},$  where $\widehat{R}$ is the completion of $R$ in its ${\bf m}$-adic topology, while $\widehat{S}$ and  $\widehat{K}$ are the
completions of $K$ in the ${\bf m}$-adic topology of $S$.
\end{cor}

\begin{proof}  Let  $M$ and $N$ denote the maximal ideals of $R$ and $S$, respectively.
   By
Theorem~\ref{prime correspond},   $N = MS$. Let $A = S \cap \Ker D$, and let ${\ff m} =M \cap A$.  By Theorem~\ref{C connection}, $A \subseteq S$ is $C$-analytic.  Thus   by Proposition~\ref{basic cor}, ${\ff m}$ is a maximal ideal of $A$, so necessarily, ${\ff m}= MS \cap A$. We claim that ${\ff m}R$ is an $M$-primary ideal of
$R$.  Indeed, if $P$ is a prime ideal of $R$ such that ${\ff m}R
\subseteq P$, then $P$ meets $C$ since $M$ meets $C$ and ${\ff m} =
M\cap A$.  Thus ${\ff m}S \subseteq PS$, and  by
Proposition~\ref{basic cor}, ${\ff m}S$ is the maximal ideal of $S$,
so ${\ff m}S = PS$.  This shows that for every prime ideal $P$ of $R$ containing ${\ff m}R$, we have $PS = {\ff m}S$, and hence by Theorem~\ref{prime correspond}, $P = PS \cap R = {\ff m}S \cap R = M$.  Therefore, $M$ is the unique prime ideal of $R$ containing ${\ff m}R$, and
 hence
${\ff m}R$ is $M$-primary.  Now since $M$ is
finitely generated, some power of $M$ is contained in ${\ff m}R$,
and hence  $\widehat{R} \cong \lim_{\leftarrow} R/{\ff m}^iR$. Similarly, since ${\ff m}S$ is the
maximal ideal of $S$, we have $\widehat{S} \cong \lim_{\leftarrow}
S/{\ff m}^iS$ and $\widehat{K} \cong \lim_{\leftarrow} K/{\ff m}^iK$.
Now let $c \in {\ff m} \cap C$.  Then by Theorem~\ref{pre-construction} we have that for each $i$, the induced  mapping $R/c^iR \rightarrow S/c^iS \star K/c^iK$ is an isomorphism. It follows that the induced mapping $R/{\ff m}^i \rightarrow S/{\ff m}^iS \star K/{\ff m}^iK$ is an isomorphism. This in turn implies that $f$ lifts to an isomorphism $\widehat{R} \rightarrow \widehat{S} \star \widehat{K}$.   
\end{proof}

\begin{rem} \label{late remark} {\em Since the maximal ideal $N$ of $S$ is extended from that of $R$, the $N$-adic and $M$-adic topologies on $S$-modules agree. Hence $\widehat{S}$ and $\widehat{K}$ can also be viewed as the completions of $S$ and $K$ in the $M$-adic topology.}
\end{rem} 

\section{Noetherian rings}
 
In this section we characterize  when strongly twisted subrings are Noetherian, and consider also a special situation when being twisted
along a multiplicatively closed subset is enough to guarantee the subring is Noetherian.  

\begin{lem} \label{pre-construction 2}  Let $S$ be a ring,  let $C$ be a multiplicatively closed subset of nonzerodivisors of $S$, and let $K$ be a $C$-torsion-free $S$-module.   Suppose that 
 $R$ is a subring of $S$ that is twisted by $K$ along $C$, and let $I$ be an ideal of $R$ meeting $C$.
  If $IS \cap R$ \index{twisted subring!finitely generated ideal of}
is a finitely generated ideal of $R$, then $K/IK$ is a finitely
generated $S$-module. Conversely, if $I = IS \cap R$,  $IS$ is a
finitely generated ideal of $S$ and $K/cK$ is a finitely generated 
$S$-module for some $c \in I \cap C$, then $I$ is a finitely
generated ideal of $R$.
\end{lem}

\begin{proof}  Let $D$ be the derivation that twists $R$, and let $A = S \cap \Ker D$.  
  Suppose that $IS \cap R$ is a finitely generated ideal of $R$, and
  write $IS \cap R = (x_1,\ldots,x_n)R$.  Observe that since $K$ is an $S$-module,
   $IK \subseteq (IS \cap R)K \subseteq
  IK$, so that $(IS \cap R)K = IK$.
     We
claim that  $K = SD(x_1) + \cdots + SD(x_n) + IK$, and we prove this indirectly.
First we show  that  $K$ is generated as an $S$-module by $D(R)$.  For if $x \in K$, then since $K_C$ is generated by $D(S_C)$ as an $S_C$-module and $D$ is $C$-linear, there exist $s_i,\sigma_i \in S$ and $c \in C$ such that  $x = \sum_i \frac{s_i}{c}D(\sigma_i)$.  Since $S/A$ is $C$-divisible, there exist $a_i \in A$ and $\tau_i \in S$ such that $\sigma_i = a_i  + c\tau_i$.  Since $D(a_i) =0$, it follows that $x = \sum_i s_iD(\tau_i)$.  Now choose $e \in C$ such that for all $i$,  $eD(\tau_i) \in K$. For each $i$, write $s_i = b_i + et_i$ for some $b_i \in A$ and $t_i \in S$.  Then $x = \sum_i s_iD(\tau_i) = D(\sum_i b_i\tau_i) + \sum_i t_iD(e\tau_i)$.  Now $e\tau_i \in D^{-1}(K) \cap S = R$, and $x \in K$, so that $\sum_i b_i\tau_i \in D^{-1}(K) \cap S = R$, which proves that $K$ is generated as an $S$-module by $D(R)$. 

Now let $y \in IS \cap R$, and write $y = x_1r_1 + \cdots +
x_nr_n$ for $r_1,\ldots,r_n \in R$. Then
\begin{eqnarray*} D(y) & = & r_1D(x_1) + \cdots + r_nD(x_n) +
x_1D(r_1) + \cdots x_nD(r_n) \\ \: & \in & RD(x_1) + \cdots +
RD(x_n) + IK.\end{eqnarray*}    Therefore, $D(IS \cap R) \subseteq SD(x_1) +
\cdots + SD(x_n) + IK$. Now
 since $S/A$ is $C$-divisible,  then $S = A + IS$, and since $A \subseteq R$, we have then $R = A + (IS \cap R)$.  Thus
 $$D(R) =
D(A)+D(IS \cap R) = D(IS \cap R).$$ 
Since, as we have shown, $K$ is generated as an
$S$-module by $D(R)=D(IS \cap R)$, we conclude that $K = SD(x_1) + \cdots +
SD(x_n) + IK$. Therefore, $K/IK$ is a finitely generated $S$-module.

Conversely, suppose that $I = IS \cap R$,  $IS$ is a finitely generated ideal of $S$
and $K/cK$ is a finitely generated $S$-module for some $c \in I \cap
C$.  Then by Theorem~\ref{pre-construction}, the mapping $f$ induces an isomorphism $f_c:R/cR \cong S/cS \star K/cK$, and by 
Lemma~\ref{one case},  $$I/cR \cong f_c(I/cR)(S/cS \star K/cK) = IS/cS \star
K/cK.$$  Now $IS/cS$ and $K/cK$ are  finitely generated
$S$-modules.  But $S = A + cS$, so it follows that $IS/cS$ and $K/cK$ are finitely generated $A$-modules.
  Therefore, $I/cR$ is a finitely
generated $A$-module, and hence $I$ is a finitely generated ideal of
$R$.
\end{proof}

\begin{thm} \label{Twisted Noetherian subrings} Suppose that $S$ is a domain and $R$ is a subring of $S$ strongly twisted by a torsion-free $S$-module $K$. Let $D$ be the derivation that strongly twists $R$.  
The ring $R$ is a Noetherian domain if and only if $S$ is a Noetherian domain and for each $0 \ne a \in S \cap \Ker D$,  $K/aK$ is a finitely generated $S$-module.
\end{thm}

\begin{proof} Let $A = S \cap \Ker D$. If $R$ is a Noetherian domain, then since every prime ideal of
$S$ is extended from $R$ (Theorem~\ref{prime correspond}), every
prime ideal of $S$ is finitely generated, and hence $S$ is
Noetherian. Moreover, if $R$ is Noetherian, then for every $0 \ne a
\in A$, $aS \cap R$ is a finitely generated ideal of $R$, and so by
Lemma~\ref{pre-construction 2}, $K/aK$ is a finitely generated
$S$-module. Conversely, if $S$ is Noetherian and $K/aK$ is a
finitely generated $S$-module for every $0\ne a \in A$, then by
Lemma~\ref{pre-construction 2} every ideal $I$ of $R$ of the
form $I = IS \cap R$ is finitely generated. Since $R \subseteq S$ is an integral extension, every prime ideal of $R$ has this form, and hence
 every
prime ideal of $R$ is finitely generated, and $R$ is Noetherian.
\end{proof}

In the setting of  Theorem~\ref{Twisted Noetherian subrings},  the Noetherian rings between $R$ and $S$  correspond by Theorem~\ref{correspondence} to the $S$-submodules $L$ of $FK$ that contain $K$ for which $L/aL$ is a finitely generated $S$-module for all $0 \ne a \in S \cap \Ker D$.  Of course, when $R$ has dimension $1$, then every ring between $R$ and $S$ must be Noetherian, but in higher dimensions, the preceding observations make it easy to find non-Noetherian rings between $R$ and $S$ when $K$ is finitely generated.  By contrast, if $K$ is not finitely generated, we see below that it can happen that there are no non-Noetherian rings between $R$ and $S$ when $S$ has dimension $>1$.  But in the case where $K$ is finitely generated, non-Noetherian rings must occur:

\begin{prop} \label{fg non} Let $S$ be a Noetherian domain of Krull dimension $>1$, and suppose $R$ is a subring of $S$ strongly twisted by a finitely generated torsion-free $S$-module $K$.  Then there exists a non-Noetherian ring between $R$ and $S$. \end{prop}

  \begin{proof}   
Define $T = S \cap D^{-1}(K_P)$.  Then $T$ is a ring with $R \subseteq T \subseteq S$ and   $T$ is strongly twisted by $K_P$.
Suppose by way of contradiction that  $T$ is a Noetherian ring. Let $0 \ne c \in P$.  Then by Theorem~\ref{Twisted Noetherian subrings}, 
 $K_P/cK_P$ is a finitely generated $S$-module.  Moreover, since $K$ is a finitely generated $S$-module and $c \in P$, Nakayama's Lemma implies that $K_P/cK_P$ is a nonzero $S$-module.    Let $E$ be a nonzero cyclic $S_P$-submodule of $K_P/cK_P$.  Then since $S$ is a Noetherian ring and $K_P/cK_P$ is a finitely generated $S$-module, $E$ is also a finitely generated $S$-module.   Now $E \cong S_P/(0:_{S_P}E)$, and since $E \ne 0$, then $S_P/PS_P$ is a homomorphic image of $E$, and hence $S_P/PS_P$ must also be a finitely generated $S$-module.  But $S_P/PS_P$ is isomorphic to the quotient field of the domain $S/P$, so the finite generation of $S_P/PS_P$ forces $S_P = S$, which in turn implies that $P$ is a maximal ideal of $S$, a contradiction.  Therefore, $T$ is a non-Noetherian ring between $R$ and $S$.  \end{proof}

  The assumption that $K/aK$ is a finitely generated $S$-module for all $0 \ne a \in S \cap \Ker D$ is  weaker than simply requiring $K$ itself to be finitely generated.  This
subtlety leaves room for an interesting class of examples where although $K$ is not finitely generated, it produces Noetherian subrings $R$ of $S$.  We illustrate this in Example~\ref{new V example}, which uses the following observation.

\begin{lem} \label{tffr lemma} Let $V$ be a DVR with maximal ideal ${\ff M}$, let  $K$ be a torsion-free finite rank
$V$-module and let $r = \dim_{V/{\ff M}}K/{\ff M} K.$  Then $r \leq \rank(K)$ and
 for all proper nonzero ideals $J$ of $V$, $K/JK$ is a free $V/J$-module of rank $r$.   \end{lem}

\begin{proof}
Let ${\cal F}$ be the set of all $V$-submodules $H$ of $K$ such that
$K/H$ is a nonzero torsion-free divisible $V$-module.  Suppose first
that  ${\cal F}$ is empty.  Since $K$ is torsion-free, we may
view $K$ as contained in an $F$-vector space $L$ of the same rank,
say $n$, as $K$.  Write $L = Fe_1 \oplus \cdots \oplus Fe_n$, where
$e_1,\ldots, e_n$ is a basis for $L$.  Then since $\cal F$ is empty,
the image of the projection map $\pi_i:K \rightarrow F$ of $K$ onto
the $i$-th coordinate is not all of $F$.  Thus since $V$ is a DVR,
there exists $t \in V$ such that $\pi_i(K) \subseteq t^{-1}V$ for
all $i=1,\ldots,n$.  Therefore, $K \subseteq t^{-1}Ve_1 \oplus
\cdots \oplus t^{-1}Ve_n$, and hence $K$ is a submodule of a
finitely generated free $V$-module.  Since $V$ is a DVR, $K$ is a
 free $V$-module of rank $r$, and hence in the case where ${\cal F}$ is empty, the lemma is proved.

Next  suppose that ${\cal F}$ is nonempty, and  let $m$ be the
maximum of the ranks of the torsion-free  divisible $R$-modules $K/H$, where $H$ ranges over the members of ${\cal F}$.
 Choose $H \in
{\cal F}$ such that $K/H$ has rank $m$.   Let $J$ be a proper nonzero ideal of $V$, and write $J = vV$ for some $v \in V$.
We claim
that $K/vK \cong H/vH$, and that $H$ is a free
$V$-module (necessarily of rank no more than the rank of $K$). Now, since $K/H$ is torsion-free, we have $vK \cap H =
vH$, and hence there is an embedding $H/vH \rightarrow K/vK$ defined
by $h + vH \mapsto h + vK$ for all $h \in K$. Moreover, since $K/H$
is a divisible $V$-module, $K = H + vK$, and hence the mapping is an
isomorphism.
If there does not exist a
$V$-submodule $G$ of $H$ such that $H/G$ is a nonzero torsion-free
divisible $V$-module, then, as above, $H$ is a free $V$-module of rank $\leq  n$. In this case, since for any $0 \ne v \in {\ff M}$, we have shown that $H/vH \cong K/vK$, it follows that $H/{\ff M}H \cong K/{\ff M}K$, so that since $H$ is a free $V$-module, $\rank(H) = r$.

Thus the only case that remains to rule out is where
 there exists a $V$-submodule $G$ of $H$ such that $H/G$ is a nonzero torsion-free divisible $V$-module.
 Assuming the existence of such a $V$-submodule $G$ of $H$, there is an exact sequence of
$V$-modules:
$$0 \rightarrow H/G \rightarrow K/G \rightarrow K/H \rightarrow 0.$$
Since $H/G$ is a divisible torsion-free $V$-module, this sequence
splits, and hence $K/G \cong H/G \oplus K/H$.  Since $H/G$ is a nonzero divisible torsion-free module and $K/H$ is  a
 divisible torsion-free $V$-module of rank $m$, this means that $K/G$
is a divisible torsion-free $V$-module of rank $>m$, contradicting
the choice of $H$, and the lemma is proved.
\end{proof}

\begin{exmp} \label{new V example} {\em Let $k$ be a field of positive characteristic  that is separably generated and of 
infinite transcendence degree over a countable subfield, let $X_1,\ldots,X_d$ be indeterminates for $k$,  and let $S = k[X_1,\ldots,X_d]_{(X_1,\ldots,X_d)}$.  Matsumura has shown that the generic formal fiber of a local domain essentially of finite type over a field has dimension one less than the domain \cite[Theorem 1]{Mat2}, and Heinzer, Rotthaus and Sally have shown that this condition on the generic formal fiber guarantees the existence of a birationally dominating DVR    having residue field finite over the residue field of the base ring \cite[Corollary 2.4]{HRS}.   Thus $S$ is birationally dominated by a DVR $V$ having residue field finite over the residue field $k$  of $S$.   
  Let $K$ be a nonzero torsion-free finite rank $V$-module that is not divisible.  By Theorem~\ref{existence strongly twisted cor}, there exists a subring $R$ of $S$ strongly twisted by $K$.  
Also, with $N$ the maximal ideal of $S$, the fact that $V$ is a DVR implies  that  $K/NK$ is a finitely generated     $V$-module (Lemma~\ref{tffr lemma}).  If $a \in S \cap \Ker D$, where $D$ is the derivation that strongly twists $R$, then since $V$ is a DVR dominating $S$ and $K$ is a $V$-module, $a^iK = N^jK$ for some $i,j\geq 0$.  Since $V$ is a DVR and the residue field of $V$ is a finitely generated $S$-module, it follows that $V/xV$ is a finitely generated $S$-module for all $0 \ne x \in V$. Thus since $K/NK$ is a finitely generated $V/NV$-module and $V/NV$ is a finitely generated $S$-module, then $K/NK$ is a  finitely generated $S$-module.  Now since $N$ is a finitely  generated ideal of $S$, it follows that $K/N^jK = K/a^iK$ is a finitely generated $S$-module.  Therefore, $K/aK$ is a finitely generated $S$-module, and by Theorem~\ref{Twisted Noetherian subrings}, $R$ is a Noetherian ring.  Moreover, if the dimension of $S$ is more than $1$ and $K=V$, then $K$ is not a finitely generated $S$-module.  Twisted subrings arising in this manner are considered later in this section and the next, as well as in \cite{OlbSub}.}
\end{exmp}

The  next propositions contrasts the sort of twisted subrings occurring in the example  
with  those  in Proposition~\ref{fg non}.

\begin{prop} Let $S$ be a local Noetherian domain with maximal ideal $N$.  Suppose  that $S$ is birationally dominated by a DVR $V$ such that $V = S + NV$, and  that  there exists a subring $R$ of $S$ that is  strongly twisted by a torsion-free finite rank $V$-module $K$. 
  Then        every ring between $R$ and $S$ is a local 
 Noetherian ring that is strongly twisted by some $V$-module $L$ with $K \subseteq L \subseteq FK$.  
\end{prop}

\begin{proof} Let $T$ be a ring such that $R \subseteq T \subsetneq S$, and let $F$ denote the quotient field of $S$.  Then by Theorem~\ref{correspondence},  there exists an $S$-module $L$ such that $K \subseteq L \subseteq FK$ and $T = S \cap D^{-1}(L)$.  We claim that $L$ is in fact a $V$-module (not just an $S$-module).  Let $E$ be a free $V$-submodule of $L$ having the same rank as $L$, and let $e_1,\ldots,e_n$ be a basis for $E$, so that $E = Ve_1 \oplus \cdots \oplus Ve_n$.  Then since $E$ and $L$
have the same rank, we may view $L$ as an $S$-submodule of $Fe_1
\oplus \cdots \oplus Fe_n$.  To show that $L$ is a $V$-module, it
suffices to show that for each $v \in V$ and $y \in L$, $vy \in E +
Sy$.  Let $v \in V$ and $y \in L$, and write $y = x_1e_1 + \cdots +
x_ne_n$, where $x_1,\ldots,x_n \in F$. Let $I = (V:_F x_1) \cap \cdots \cap (V:_F x_n)$.
Since $V$ is an overring of $S$, there exists $0 \ne t \in I \cap N$.  Moreover, since $V = S + NV$, it follows that $V = S + N^jV$ for all $j>0$.  Hence, since $V$ is a DVR and $0\ne t \in N$, it must be that
$V = S +
tV$.  Thus there exists $s \in S$ such that $v
-s \in tV \subseteq I$.  Consequently, $vx_i - sx_i \in V$ for all
$i=1,2,\ldots,n$. Returning now to the claim that $vy \in E + Sy$,
observe that $vy - sy = (vx_1-sx_1)e_1 + \cdots +(vx_n-sx_n)e_n \in
E$, so the claim is proved.  Therefore, $L$ is a $V$-module. By Lemma~\ref{tffr lemma}, $L/sL$ is a finitely generated $V$-module for all $0 \ne s \in S$. Since $V= S + NV$, then $V/NV$ is a cyclic $S$-module. This, along with the fact that $V$ is a DVR, implies that $V/sV$ is a finitely generated $S$-module for all $0 \ne s \in S$. Therefore, $L/sL$ is a finitely generated $S$-module for all $0 \ne s \in S$, and hence 
  by Theorem~\ref{Twisted Noetherian subrings}, $T$ is a local Noetherian domain. 
\end{proof}

The characterization of when strongly twisted subrings are Noetherian in Theorem~\ref{Twisted Noetherian subrings} depends on the subring being strongly twisted rather than twisted along $C$.  However, there is a specific circumstance when being twisted along a multiplicatively closed subset suffices to give Noetherianness.  The idea behind the following theorem originates with Ferrand and Raynaud \cite[Proposition 3.3]{FR} and the version we give here is a  generalization of a result of  Goodearl and Lenagan \cite[Proposition 7]{GL}.  Our formulation and approach to the theorem are different, but ultimately, as we point out in the proof, a key step depends on an argument from \cite{FR}.    Also, unlike  the strongly twisted Noetherian rings produced using Theorems~\ref{existence strongly twisted cor} and~\ref{Twisted Noetherian subrings}, the next theorem produces examples in arbitrary characteristic, as Corollary~\ref{existence V cor} illustrates.    



\begin{thm} \label{existence V} Let $(S,N)$ be a two-dimensional local Noetherian UFD  that  is birationally   dominated by a DVR $V$ having the same residue field as $S$, and such that there is $t \in N$ with  $tV$ the maximal ideal of $V$.  If  $R$ is a subring of $S$ that is twisted along $C =\{t^i:i>0\}$ by    an $S$-submodule $K$ of a  finitely generated free $\widehat{V}$-module $K'$ with $K'/K$ a $C$-torsion-free $S$-module, then $R$ is a two-dimensional local Noetherian ring such that for every height $1$ prime ideal $P$ of $R$, $R_P = S_Q$ for some height $1$ prime ideal $Q$ of $S$. 
\end{thm}

\begin{proof} Let $K$ be the $S$-submodule of $\widehat{V}$ by which $R$ is twisted, and   let $D$ be the derivation that twists $R$. Let $A = S \cap \Ker D$, and let $Q$ be the quotient field of $A$.      
   Since by Theorem~\ref{pre-construction c},  $R \subseteq S$ is an integral extension and $S$ is local, $R$ has a unique maximal ideal $M := N \cap R$.    To prove that $R$ is Noetherian, we show that every prime ideal of $R$ is finitely generated. 
First we claim that every ideal of $R$ containing a power of $t$ is finitely generated.  Indeed, for each $i>0$, we can argue as in Example~\ref{new V example} that 
 $V/t^iV$ is a finitely generated $S$-module.  Now since $K'/K$ is $C$-torsion-free, we have for each $i>0$, $t^iK' \cap K = t^iK$, and hence $K/t^iK \cong (K + t^iK')/t^iK' \subseteq K'/t^iK'$.  But $K'/t^kK'$ is a finitely generated  module over $\widehat{V}/t^i\widehat{V} \cong V/t^iV$, and since $V = S + tV$, it follows that $K'/t^iK'$ is a finitely generated $S$-module.  Hence $K/t^iK$, since it is isomorphic to an $S$-submodule of $K'/t^iK'$, is also a finitely generated $S$-module.  
Now  by Theorem~\ref{pre-construction}, $R/t^iR$ is isomorphic as a ring to $S/t^iS \star K/t^iK$, so that since $S/t^iS$ is a Noetherian ring and $K/t^iK$ is a finitely generated $S$-module, $R/t^iR$ is a Noetherian ring. Therefore, it follows that every ideal of $R$ containing a power of $t$ is finitely generated.  
In particular,  $M$ is a finitely generated ideal of $R$.

Since $R \subseteq S$ is an integral extension, $R$ has Krull dimension $2$, and hence to prove that $R$ is Noetherian, all that is left to show is that each height $1$ prime ideal of $R$ is finitely generated. 
 In fact, if $P$ is a height $1$  prime ideal of $R$, then since $R
\subseteq S$ is integral, there is a height $1$ prime ideal of $S$
lying over $P$.  Thus  since $S$ is a UFD, there exists $f \in S$
such that $P = fS \cap R$, and so to complete the proof of the theorem it is enough to show that $fS \cap R$ is a finitely generated ideal of $R$.  
 Our proof of this fact  is adapted from 
Ferrand and Raynaud \cite[Proposition 3.3]{FR}.

 Define $I = \{s \in S:fs \in R\}$. Then
$I$ is a fractional ideal of $R$ such that $fS \cap R= fI$.  Thus to prove
that $fS \cap R$ is a finitely generated ideal of $R$, it suffices to show
that $I$ is a finitely generated fractional ideal of $R$. %
 Now $D(f) \in K_C$, so 
 there exists $c \in C$ such that $D(cf) = cD(f) \in K$,
 and hence $cf \in R$.  Consequently, $c \in I$, and so if there exists $
b \in C$ such that  $bI \subseteq R$, then $bc \in bI \subseteq R$,
so that by what we have established above, since $bI$ contains an element of $C$, the ideal
$bI$, and hence the fractional ideal $I$, is finitely generated.
Thus it remains to show that there exists $b \in C$ such that $bI
\subseteq R$.



Let $e_1,\ldots,e_n$ be a basis for the free $\widehat{V}$-module $K'$, so that $K \subseteq \widehat{V}e_1 \oplus \cdots \oplus \widehat{V}e_n$.
For each $i =1,\ldots,n$, let $\pi_i$ be the projection of $K$ onto the $i$-th coordinate  of this direct
sum; i.e., $\pi_i(v_1e_1 + \cdots v_ne_n) = v_i$ for all $v_1,\ldots,v_n \in \widehat{V}$.  Let $s \in I$. Then $sf \in R$, so that
$sD(f) + fD(s) = D(sf) \in K$. Since $D(f) \in K_C$, 
 there exists $c \in C$ such that $cD(f) \in K$. Thus since $csD(f) + cfD(s) \in K$ and
$cD(f) \in K$, we have that $cfD(s) \in K$. Since the choice of $s\in I$ was arbitrary, we have for each $i=1,\ldots,n$ and $s \in I$ that $\pi_i(D(s)) \in
(cf)^{-1}\widehat{V}$.   
Since $t\widehat{V}$ is the maximal ideal of $\widehat{V}$,  there
exists $b \in C$ such that $b(cf)^{-1} \in \widehat{V}$. Consequently, for each $i=1,\ldots,n$ and $s \in I$,
$\pi_i(D(bs)) = b\pi_i(D(s)) \in b(cf)^{-1}\widehat{V} \subseteq \widehat{V}$, and hence for all $s \in I$,  $D(bs) \in K' \cap K_C = K$, where this last assertion follows from the fact that $K'/K$ is $C$-torsion-free.  
  This then implies that $bI \subseteq S \cap D^{-1}(K) = R$, which
proves the claim, and hence verifies that $fS \cap R$ is a finitely generated ideal of $R$.  
We conclude that $R$ is a Noetherian domain.

To prove the final assertion, let $P$ be a height $1$ prime ideal of $R$.  Since $K'/K$ is $C$-torsion-free, $K = K' \cap K_C$, so that $K_P = K'_P \cap (K_C)_P$.  Now $R_P \not \subseteq V$, for since $R_P$ has dimension $1$, the maximal ideal  of any overring of $R_P$ other than $F$ must contract in $R$ to $P$, yet the maximal ideal of $V$ contracts to $M$.  Thus  $R_P \not \subseteq V$,  and since $V$ is a DVR, it must be that $V_P$ is the quotient field of $V$, and hence $K'_P$ is a vector space over the quotient field of $\widehat{V}$.  Therefore, $K_P = (K_C)_P$.  By Lemma~\ref{ontoness}, $K_C/K$ is isomorphic as an $R$-module to $S/R$, so it follows that $S \subseteq R_P$, and hence $R_P$ is a localization of $S$ at a height $1$ prime ideal of $S$.    
\end{proof}

With this theorem and the existence result, Lemma~\ref{dig}, we give an  example in characteristic $0$ of a Noetherian twisted subring of dimension $>1$:

\begin{cor}  \label{existence V cor}  Let $k$ be a field of characteristic $0$ that has infinite transcendence degree over its prime subfield, let $X$ and $Y$ be indeterminates for $k$, and let $S = k[X,Y]_{(X,Y)}$.  
  Then for each $n \geq 1$, there exists an analytically ramified local Noetherian ring  $R$ having normalization $S$,  embedding dimension $2+n$, multiplicity $1$, and an isolated singularity.  
 \end{cor}

\begin{proof}  Let $F$ denote the quotient field of $S$.  Since $k((X))$ has infinite transcendence degree over $k$, it follows that $S$ embeds into $k[[X]]$ in such a way that the image of $(X,Y)$ is contained in $Xk[[X]]$; see \cite[p.~220]{ZS}.   Viewing $S$ as a subring of $k[[X]]$, let $V =  k[[X]] \cap k(X,Y)$. Then $V$ is a DVR with residue field $k$ and   maximal ideal generated by $X$.  
  Let $K$ be a rank $n$ free $V$-module.
  Then by Lemma~\ref{dig}, there exists a ring $A$ such that $A \subseteq S$ is a $C$-analytic extension with $C = \{X^i:i\geq 1\}$, and   there exists an $A$-linear derivation $D:F \rightarrow FK$ such
that $D(S) = FK$.  Let $R = D^{-1}(K) \cap S$.  Then $R$ is twisted by $K$ along $C$, and since the cokernel of the canonical mapping $K \rightarrow \widehat{V} \otimes_V K$ is $C$-torsion-free, then 
 by Theorem~\ref{existence V}, $R$ is a Noetherian ring.  By Theorem~\ref{pre-construction c},  $R \subseteq S$ is integral, so that also $R$ is a local ring. By Corollary~\ref{pre-construction completion}, the $M$-adic completion of $R$ is $\widehat{R} = \widehat{S} \star \widehat{K}$.  Since $N/N^2$ has dimension $2$ as an $S/N$-vector space, while  $\widehat{K}/N\widehat{K}$ has dimension $n$, it follows that $\widehat{S} \star \widehat{K}$ has embedding dimension $2+n$.  Similar calculations show that the multiplicity of $\widehat{S} \star \widehat{K}$ is $1$; we omit these calculations because in \cite{OlbSub} we describe the Hilbert polynomials and multiplicity of twisted subrings in detail, and from this description the multiplicity in the present context can easily be deduced.  Therefore, since embedding dimension and multiplicity are invariant under completion, the embedding dimension of $R$ is $2+n$ and the multiplicity of $R$ is $1$.  By Theorem~\ref{existence V}, each localization of $R$ at a height $1$ prime ideal is a DVR, so $R$ has an isolated singularity.
\end{proof}

As another application of  the theorem, we reframe   an example due to  Goodearl and Lenagan.

\begin{exmp}\label{FR theorem 2 example} {\em (Goodearl and Lenagan \cite[p.~494]{GL}) Let $k$ be a field, and let $U = k[[x]]$, with $x$ an indeterminate for $k$.
  Choose $y,z
\in xU$ such that $y$ and $z$ are algebraically independent over
$k(x)$ (see~\cite[p.~220]{ZS} for a constructive  argument that
there are infinitely many such choices for $y$ and $z$).  Let $A =  k[x,y]_{(x,y)}$,
 $W = k(x,z) \cap U$,
  $S = W[y]_{(x,y)}$, and
 $F = k(x,y,z)$.
The definition of $S$ makes sense, since $W$ is a DVR with quotient field $k(x,z)$ whose maximal ideal is $k(x,z) \cap xU$, and  since $xU$ is the maximal ideal of $U$, it follows that $xW$ is the maximal ideal of $W$.   Thus $(x,y)W[y]$ is a maximal ideal of $W[y]$.  Moreover, this shows also that $S$ is a regular local ring of Krull dimension $2$ with quotient field $F = k(x,y,z)$.   
With $C = \{x^i:i>0\}$, the extension $A \subseteq S$ is $C$-analytic.  
Now let $D=\frac{\partial}{\partial z}$, and note that $D$ is $A$-linear.   Let $L$ be the $S_C$-submodule of $k((x))$ generated by $D(S)$.  Define $K = L \cap U$ and   
 $R = S \cap D^{-1}(K)$.  Then since $U$ is a DVR with maximal ideal $xU$, it follows that $K_C = L \cap U_C = L$.  Therefore, $R$ is twisted by $K$ along $C$.  Moreover, since $K_C = L$, we have  $U \cap K_C = U \cap L = K$, and hence $U/K$ is a $C$-torsion-free  $S$-module.  Thus by Theorem~\ref{existence V}, $R$ is a local Noetherian ring having the properties of the theorem.}    
 \end{exmp}



In the example,
$k[x,y,z] \subseteq R \subseteq S \subseteq k(x,y,z)$, and since $x,y$ and $z$ are algebraically independent, we can work backwards from any field of the form $k(X,Y,Z)$ and view $k[X,Y,Z]$ as embedded in $k[[X]]$, with $Y,Z \in Xk[[X]]$.  Then we obtain a ring $R$ as in the example:

\begin{cor} \label{dimension 2 in nature} Let $k$ be a field, and let $X,Y,Z$ be indeterminates for $k$.  Then there exists a two-dimensional analytically ramified local Noetherian domain between $k[X,Y,Z]$ and $k(X,Y,Z)$ having an isolated singularity and  normalization a regular local ring.  
 \qed
\end{cor}


\section{Cohen-Macaulay rings}

We consider now when strongly twisted local subrings are Cohen-Macaulay, Gorenstein, a complete intersection or a hypersurface.  Since these properties are invariant under completion, the following theorem is a consequence of well-known facts applied to the idealization $\widehat{R} \cong \widehat{S} \star \widehat{K}$ given by Corollary~\ref{pre-construction completion}. 

\begin{thm} \label{CM} Let $S$ be a quasilocal domain, and let $R$ be a subring of $S$ strongly twisted by a finitely generated torsion-free $S$-module  $K$. 
Then the following statements hold for $R$.

\begin{itemize} 

\item[{(1)}]  $R$ is a Cohen-Macaulay ring if and only if  $S$ is a Cohen-Macaulay ring and $K$ is a finitely generated maximal Cohen-Macaulay  module.

\item[{(2)}]  $R$ is a Gorenstein ring if and only if  $S$ is a Cohen-Macaulay ring that admits a canonical module $\omega_{S}$ and $K \cong \omega_S$.  

\item[{(3)}] If $S$ is a Gorenstein ring, then for $K =S$,  the ring $R$ is a Gorenstein ring.

\item[{(4)}]  $R$ is a complete intersection if and only if $S$ is a complete intersection and $K \cong S$.

\item[{(5)}]  $R$ is a hypersurface if and only if $S$ is a regular local ring and $K \cong S$.

\end{itemize}
\end{thm} 

\begin{proof}
First observe that since $K$ is finitely generated, then by Theorem~\ref{Twisted Noetherian subrings}, $R$ is Noetherian if and only if $S$ is Noetherian. 

(1) 
Since a local ring is Cohen-Macaulay if and only if its completion  is Cohen-Macaulay, it is enough to determine when $\widehat{R}$ is Cohen-Macaulay \cite[Corollary 2.1.8, p.~60]{BH}.  But by Corollary~\ref{pre-construction completion}, $\widehat{R} \cong \widehat{S} \star \widehat{K}$, so this is easy to do.  Indeed,    properties of idealizations show  that $\widehat{R}$ is Cohen-Macaulay if and only if $\widehat{S}$ is Cohen-Macaulay and $\widehat{K}$ is a \index{maximal Cohen-Macaulay module} {\it maximal Cohen-Macaulay $\widehat{S}$-module} (meaning that the depth of $\widehat{K}$, its dimension and the dimension of $\widehat{S}$ are all the same); see \cite[Corollary 4.14]{AW} or \cite[p.~52]{Val}.  Since $K$ is a finitely generated torsion-free $S$-module, $K$ is a maximal Cohen-Macaulay $S$-module if and only if $\widehat{K}$ is a maximal Cohen-Macaulay module \cite[Corollary 2.1.8, p.~60]{BH}.

(2) A Cohen-Macaulay ring admits a canonical module if and only if it is the homomorphic image of a Gorenstein ring \cite[Theorem 3.3.6]{BH}.  We collect several other facts: 
(a)  a local ring is Gorenstein if and only if its completion is Gorenstein \cite[Proposition 3.1.19, p.~95]{BH}; 
(b)  $\widehat{\omega_S} = \omega_{\widehat{S}}$ \cite[Theorem 3.3.5, p.~110]{BH}; (c)  since $K$ and $\omega_{S}$ are finitely generated torsion-free modules, $\widehat{K} \cong \widehat{\omega_{S}}$ if and only if $K \cong \omega_{S}$ (this can be deduced for example from Theorems 7.5(i) and 8.11 of \cite{Ma}); and (d) when $A$ is a local Noetherian ring and $M$ is an $A$-module, then $A \star M$ is a Gorenstein ring if and only if $A$ admits a canonical module $\omega_A$ and $M \cong \omega_A$ (apply \cite[Theorem 7]{Reiten} and \cite[p.~52]{Val}).
Thus, combining these observations with the fact that $\widehat{R} \cong \widehat{S} \star \widehat{K}$,
we have that $R$ is Gorenstein if and only if $\widehat{S} \star \widehat{K}$ is Gorenstein; if and only if $S$ admits a canonical module and $K \cong \omega_{S}$.

(3)  The ring $S$ is Gorenstein if and only if $S$ is Cohen-Macaulay and  $\omega_S \cong S$ \cite[Theorem 3.3.7, p.~112]{BH}.  Now apply (2).

(4)  Given a local ring $A$ and $A$-module $M$, the local ring $A\star M$ is a complete intersection if and only if $A$ is a complete intersection and $M \cong A$ \cite[p.~52]{Val}.  Thus since a local ring is clearly a complete intersection if and only if its completion is a complete intersection, we may use the fact that $\widehat{R} \cong \widehat{S} \star {\widehat{K}}$ to obtain (4).  

(5) Whether $R$ is a \index{hypersurface} hypersurface (meaning that $\widehat{R}$ is isomorphic to regular local ring modulo a principal ideal) is deduced from \cite[p.~52]{Val}:  With $A$ a local ring and $M$ an $A$-module, the ring $A \star M$ is a hypersurface if and only if $A$ is  a regular local ring and $M \cong A$.  
\end{proof}

We can also use these ideas to find, for example, all the Cohen-Macaulay rings between the rings $R$ and $S$ when $S$ is Cohen-Macaulay.   

\begin{cor}  If $S$ is a local Cohen-Macaulay domain with quotient field $F$,  $K$ is a finitely generated  torsion-free $S$-module, and  $R$ is strongly twisted by $K$, then the Cohen-Macaulay rings properly  between $R$ and $S$ are in one-to-one correspondence with the maximal Cohen-Macaulay modules properly between $K$ and $FK$.  The correspondence is given by the derivation that twists $R$, as in Theorem~\ref{correspondence}.  Moreover, there exists a complete intersection between $R$ and $S$ if and only if $\rank \: K = 1$.  
\end{cor}

\begin{proof}  If $T$ is a  ring between $R$ and $S$, then by Theorem~\ref{correspondence}, $T$ is strongly twisted by a unique $S$-module $L$ with $K \subseteq L \subseteq FK$.  By Theorem~\ref{CM}, $T$ is a Cohen-Macaulay ring if and only if  $L$ is a maximal Cohen-Macaulay $S$-module.  Moreover, by the theorem, $T$ is a complete intersection if and only if $L$ is a rank one free  $S$-module.  If $\rank \: K = 1$, then since $K$ is a finitely generated $S$-module, there is a rank one free $S$-module between $K$ and $FK$, and hence there is a complete intersection between $R$ and $S$.   
\end{proof}

Next we consider the case where $S$ is a local Noetherian domain and  there exists at least one proper subring of $S$ that is strongly twisted by an $S$-module.  From Proposition~\ref{very new} this then leads to an abundance of Noetherian subrings strongly twisted by $S$-modules, namely one for each fractional ideal of $S$.

\begin{thm} \label{convex examples} 
Let $S$ be a local Noetherian domain having a strongly twisted subring, and let $D:F \rightarrow F$ be the derivation given by Proposition~\ref{very new}(2).  
  For each fractional ideal $I$ of $S$, let $R_I=  S \cap D^{-1}(I)$, so that $R_I$ is the subring of $S$ strongly twisted by $I$. 
Then: 

\begin{itemize}

\item[{(1)}]  The set of all $R_I$ is convex: If $I$ and $J$ are fractional ideals of $S$, and $T$ is a ring with $R_I \subseteq T \subseteq R_J$, then there exists a fractional ideal $K$ of $T$ such that $T = R_{K}$.

\item[{(2)}]  If $S$ has Krull dimension $>1$, then for each fractional ideal $I$ of $S$, there is a non-Noetherian ring between $R_I$ and $S$.

\item[{(3)}] The set of $R_I$ forms a lattice (without top or bottom element): For each pair of fractional ideals $I$ and $J$ of $S$, $R_{I+J} = R_{I} + R_{J}$; $R_{I \cap J} = R_I \cap R_{J}$; and $I \subseteq J$ if and only if $R_{I} \subseteq R_{J}$

\item[{(4)}]  The ring $R_I$ is a Cohen-Macaulay ring if and only if $S$ is Cohen-Macaulay and the fractional ideal $I$ is a maximal Cohen-Macaulay $S$-module.  Thus when $S$ is Cohen-Macaulay, then for each $N$-primary ideal $I$ of $S$, $R_I$ is a Cohen-Macaulay ring.

\item[{(5)}] Suppose that $S$ is a Cohen-Macaulay ring  that admits a canonical module $\omega_{S}$ (which is necessarily isomorphic to an ideal of $S$).  Then for each fractional ideal $I$ of $S$,   the ring $R_I$ is a Gorenstein ring if and only if $I \cong \omega_S$.

\item[{(6)}] The ring $R_I$ is a complete intersection (resp., hypersurface) if and only if $S$ is a complete intersection (resp., regular local ring) and $I$ is a principal fractional ideal of $S$.

\end{itemize}

\end{thm}

\begin{proof}  (1) Let $K$ be the $S$-submodule  of $F$ generated by $D(T)$.  Then by Theorem~\ref{correspondence}, 
 $T=D^{-1}(K) \cap S$.   Also, $D(R_I) \subseteq D(T) \subseteq D(R_J)$, and again by Theorem~\ref{correspondence}, 
 $I$ is generated as an $S$-module by $D(R_I)$, while $J$ is generated as an $S$-module by $D(R_J)$.  Thus   $I \subseteq K \subseteq J$, so that  $K$ is a fractional ideal of $S$ with $T = R_K$.
 
 (2) Apply Proposition~\ref{fg non}.  
 
 (3)  Since $R_I = D^{-1}(I) \cap S$ and, as noted in the proof of (1), $I$ is the $S$-submodule of $F$ generated by $D(R_I)$, it follows that
 $I \subseteq J$ if and only if $R_I \subseteq R_J$.
 Thus $R_{I}+R_{J} \subseteq R_{I+J}$.  Also,  since $R_I + R_J$ is a ring (Theorem~\ref{pre-construction c}), we have 
  by (2) that  $R_I + R_J = R_K$ for some fractional ideal $K$ with $I+J \subseteq K$.  On the other hand, since $R_K \subseteq R_{I+J}$, it must be that $K \subseteq I+J$, and hence $K = I +J$.

(4), (5) and (6):  Apply Theorem~\ref{CM}.  
 \end{proof}

We return now to the case considered at the end of Section 5 in which $R$ is strongly twisted by a $V$-module, where $V$ is a DVR overring of $S$.  We see below that this case never produces Cohen-Macaulay rings, except in dimension $1$.  
First we note that the fact that $K$ is a $V$-module has an interesting consequence for $\Spec(R)$.  Everywhere off the closed point $\{M\}$ of $\Spec(R)$, the local rings of the points of $\Spec(R)$ and $\Spec(S)$ are the same.

\begin{prop}  \label{last minute prop}
Let $S$ be a local Noetherian domain, and suppose that there exists a DVR $V$ birationally dominating $S$ and having residue field finite over $S$.  If $K$ is a nonzero torsion-free finite rank  $V$-module and $R$ is a subring of $S$ strongly twisted by $K$, then $R$ is a local Noetherian domain and  
for each nonmaximal prime ideal $P$ of $R$, $R_P = S_{P'}$, where $P'$ is the unique prime ideal of $S$ lying over $R$.
  
\end{prop}

\begin{proof} 
An argument such as that in Example~\ref{new V example} shows that $R$ is a local Noetherian domain.
Let $P$ be a nonmaximal prime ideal of $R$.
Then since the maximal ideal of $V$ contracts to $M$ and $V$ is a DVR, it must be that
that   $V_P= F$, and hence since $K$ is a $V$-module, $K_P = FK$.  Consequently, by  Lemma~\ref{ontoness}, $R_P = S_P$.
 Let $P'$ be a prime ideal of $S$ lying over $P$.  Since $P' \cap S = P$, we have that $P'S_P \ne S_P$, and hence $P'S_P = PS_P$.
To see  that this implies $S_{P'} \subseteq S_P$, let $x \in S_{P'}$.  Then $S \cap x^{-1}S \not \subseteq P'$.  If $x \not \in S_P$, then  $R \cap x^{-1}S \subseteq P$.  But then $S \cap x^{-1}S \subseteq R_P \cap x^{-1}S_P \subseteq PR_P = P'S_P$, so that $S \cap x^{-1}S \subseteq P'S_P \cap S  = P'$, a contradiction that implies $S_{P'} = S_P = R_P$.  Thus the proposition is proved.\end{proof}

\begin{cor} \label{isolated cor} With $R$ and $S$ as in the proposition, if $S$ is a regular local ring, then $R$ has an isolated singularity. \qed 
\end{cor}

It follows from the proposition that if $S$ has Krull dimension $>1$ and $K \ne FK$ (so that $R \subsetneq S$), then $R$ is not Cohen-Macaulay   (compare to Theorem~\ref{CM}).  Certainly if it was, then $S$ could not be integrally closed, since  unmixedness would force  Serre's condition $S_2$ on $R$, which, along with  the regularity condition $R_1$ on $S$, and hence $R$,  would imply $R$ is integrally closed, contradicting the fact that $R \subsetneq S$ is integral.    But regardless of whether $S$ is integrally closed, unmixedness fails in a strong way for $R$ when $S$ has Krull dimension $>1$:

\begin{prop} \label{divisorial} With $R$, $K$ and $S$ as in Proposition~\ref{last minute prop} and $K \ne FK$,  the maximal ideal   $M$ of $R$ is the associated prime  of a nonzero principal ideal.
\end{prop}

\begin{proof}   It suffices to exhibit an element $s \in (R:_F M)$ that is not in $R$, for then $M = R \cap s^{-1}R$ and the proposition follows.
 Let $t \in M$ such that $tV = MV$.  Observe that $tK \ne K$, for otherwise since $V$ is a DVR and $tV \ne V$, it follows that $K$ is a divisible $V$-module and hence $K = FK$, contrary to  assumption.
 Therefore,
$K \subsetneq t^{-1}K \subseteq FK$, and since $R = D^{-1}(K) \cap S$ and by Theorem~\ref{correspondence},  $R \ne D^{-1}(t^{-1}K) \cap S$,   there exists $s \in S$ such that $D(s) \in
t^{-1}K \setminus K$.  (Here, $D$ is the derivation that twists $R$.)  Now, since $D$ is a derivation, $K$ is a $V$-module and $MV = tV$, we have  $$D(sM) = sD(M) + MD(s) \subseteq K
+Mt^{-1}K = K.$$  Thus
$sM \subseteq D^{-1}(K) \cap S = R$, and we have $s \in (R:_FM)
\setminus R$, as claimed.  \end{proof}

If $A$ is a local Cohen-Macaulay ring of Krull dimension $d$, then an inequality due to Abhyankar in  \cite{AbhLocal} places a lower bound on the multiplicity $e(A)$ of $A$: $$e(A) \geq \embdim A - d+1.$$ To contrast this with the non-Cohen-Macaulay case, 
Abyhankar constructs in \cite{AbhLocal} for each pair of integers $n > d>1$ a local ring of  embedding dimension $n$, Krull dimension $d$ and  multiplicty $2$.  Example~\ref{new V example} can be used to accomplish something similar:

\begin{exmp} {\em
Let $n>d > 1$, and let $S = k[X_1,\ldots,X_d]_{(X_1,\ldots,X_d)}$ and  $V$ be as in Example~\ref{new V example}.  Let $K$  be a free $V$-module of rank $n$, and let $R$ be the subring of $S$ that is strongly twisted by $K$.  Then, as in the example, $R$ is a local Noetherian domain. Since $R \subseteq S$ is an integral extension, $R$ has Krull dimension $d$.  Moreover, as in Corollary~\ref{existence V cor}, the fact that $\widehat{R}$ is isomorphic as a ring to $ \widehat{S} \star \widehat{K}$ implies that $R$ has multiplicty $1$ and embedding dimension $d+n$.  Also, by Corollary~\ref{isolated cor}, $R$ has an isolated singularity, and by Proposition~\ref{divisorial}, the maximal ideal of $R$ is associated to a principal ideal of $R$. }\qed       
\end{exmp}











\section{Non-Noetherian rings}

Although our focus is mainly on the Noetherian case, we make in this section a  few remarks on twisted subrings of not-necessarily-Noetherian domains.
Specifically, we   characterize   the twisted subrings  of $S$,
where $S$  is
allowed to be either a Pr\"ufer domain, a Dedekind domain or a Krull
domain, and we see that various degrees of ``stability'' are necessitated
by such assumptions on $S$. Following Lipman \cite{Lipman} and Sally and Vasconcelos \cite{SV}, an ideal $I$ of a domain $R$ is {\it stable} if $I$ is projective over its ring of endomorphisms.  In case $R$ is quasilocal, $I$ is stable if and only if $I^2 = iI$ for some $i \in I$ \cite[Lemma 3.1]{OlStructure}.  A domain is {\it finitely stable} if every nonzero finitely generated ideal is stable; it is {\it stable} if every ideal is stable. 
We use the following two facts; the first  is due to Rush 
\cite[Theorem 2.3]{Rush1}, and the second, which can be found in \cite[Corollary 2.5]{OlClass},  is based on similar ideas.  

\smallskip

{(a)}
If $R$ is a finitely stable domain, then $R \subseteq
\overline{R}$ is a quadratic extension and $\overline{R}$ is a Pr\"ufer domain.
Conversely, if $R \subseteq S$ is a quadratic extension and $\overline{R}$ is a
Pr\"ufer domain such that at most two maximal ideals of $\overline{R}$ lie over
each maximal ideal of $R$, then $R$ is a finitely stable domain.

\smallskip

{(b)} 
 A  domain $R$  is one-dimensional and stable   if and only if  $R \subseteq \overline{R}$ is a quadratic extension;  $\overline{R}$ is a Dedekind domain; and
there are at most two maximal ideals of $\overline{R}$ lying over each maximal
ideal of $R$.

\smallskip

\begin{thm} \label{fs} Let $S$ be an integrally closed domain, and let $C$ be a multiplicatively closed subset of $S$.  Suppose that $R$ is a subring of $S$ that is twisted along $C$ by some $C$-torsion-free $S$-module.
 Then:

\begin{itemize}

\item[{(1)}]  $S$ is a Pr\"ufer
domain if and only if $R$ is a finitely stable domain.

\item[{(2)}]   $S$ is a Dedekind
domain if and only if   $R$ is a stable domain of Krull dimension
$1$.

\item[{(3)}]   $S$ is a Krull
domain if and only if  $S$ is the
intersection of its localizations at height $1$ prime ideals; the
set of height one prime ideals  of $R$ has finite character; and for
each such prime ideal $P$, $R_P$ is a stable domain.

\end{itemize}
\end{thm}

\begin{proof}
First note that
 by Theorem~\ref{pre-construction c},
  $R \subseteq S$ is a
quadratic extension and  $R$ and $S$ share the same quotient field.  
 By Theorem~\ref{C-bijection} every prime ideal of $R$ has a unique prime ideal of $S$ lying over it.   Moreover, since $R \subseteq S$ is integral and $S$ is integrally closed, the ring $S$ is the integral closure of $R$ in its quotient field.
Thus  to prove (1) we may apply
(a) above to obtain that $R$ is a finitely stable domain if and
only if $S$ is a Pr\"ufer domain.
 Moreover, by (b),
  $S$ is a Dedekind domain if and only if  $R$ is a stable
domain of Krull dimension $1$, and this proves (2).

To prove (3), observe first that since
 each height $1$ prime ideal of $R$ has a unique height $1$ prime ideal of $S$ lying over it, it follows that the set of height $1$ prime ideals of $S$ has finite character if and only if the set of height $1$ prime ideals of $R$ has finite character.
Suppose that $S$ is a Krull domain, and
 let $P$ be a height $1$ prime ideal of $S$.  We claim that $S_{P \cap R} = S_P$.  Indeed, since $S$ is a Krull domain, $S = \bigcap_{Q}S_Q$, where $Q$ ranges over the height $1$ prime ideals of $S$.  Since this intersection has finite character, it follows that $S_{P \cap R} = \bigcap_{Q} (S_Q)_{P \cap R}$.  Since $S_Q$ is a DVR and there is a unique prime ideal of $S$ lying over $P \cap R$, then $(S_Q)_{P \cap R}$ is the quotient field of $S$ for all $Q \ne P$.  Thus $S_{P \cap R} = S_P$, and from the fact that $R \subseteq S$ is a quadratic extension, we obtain that for each height $1$ prime ideal $P$ of $S$,
 $R_{P \cap R} \subseteq S_P$ is a quadratic
extension.
Therefore,  by (b) above, $R_{P
\cap R}$ is a stable domain.

Conversely, suppose that  $S$ is the
intersection of its localizations at height $1$ prime ideals; the
set of height 1 prime ideals  of $R$ has finite character; and for
each such prime $P$, $R_P$ is a stable domain.  Then, as we have already noted, the set of height $1$ prime ideals of $S$ has finite character, so it remains to show that $S_{P}$ is a DVR for each prime ideal $P$ of $S$.  Now by assumption $R_{P \cap R}$ is a stable domain of Krull dimension $1$, and hence by (b), the
 integral closure of $R_{P \cap R}$ in its quotient field is  a Dedekind domain.  But the quasilocal domain $S_P$, as an integrally closed overring of $R_{P \cap R}$, must contain this Dedekind domain and hence $S_P$ must be a DVR.
 Thus
$S$ is a Krull domain.
\end{proof}




As an example of how to apply Theorem~\ref{fs}, as well as Theorem~\ref{existence strongly twisted cor} (the theorem on the existence of strongly twisted subrings), we   build in Corollary~\ref{global stable example} a one-dimensional  stable
domain that has infinitely many maximal ideals $M_n$, each of which has a generating set of  prescribed size.   The existence of such rings is a consequence of a general fact, which we establish in Proposition~\ref{global stable}, regarding Dedekind domains that have a strongly twisted subring.  The proposition relies on the following technical observation.

\begin{lem} \label{construction}
Let $S$ be a domain with quotient field $F$, and let $K$ be a nonzero torsion-free $S$-module.  If $R$ is a subring of $S$ that is strongly twisted by $K$, then  for each nonzero prime ideal $P$ of $S$, the subring  $R_{P \cap R}$ of $S_P$   is strongly twisted by $K_P$.

\end{lem}

\begin{proof}  
 Let $D$ be the derivation that twists $R$, let $A = S \cap \Ker D$, and note that by Corollary~\ref{connection}, $A \subseteq S$ is a strongly analytic extension.  First we show that $S_P = S_{P \cap A}$, where the second localization is with respect to $A \setminus (P \cap A)$.  We need only verify that $S_P \subseteq S_{P
\cap A}$, since the reverse inclusion is clear. In fact, it suffices to verify that $s^{-1} \in S_{P \cap A}$ for each $s \in S \setminus P$.  To this end, let $s \in S \setminus P$.  Then $s^{-1} \in S_{P \cap A}$ if and only if $A \cap sS \not \subseteq P$.  If $A \cap sS \subseteq P$, then applying 
Proposition~\ref{basic cor}(1) we have  $sS = (A \cap sS)S \subseteq P$, contrary to the choice of $s$.  Hence $A \cap sS \not \subseteq P$, and the claim that $S_P  = A_{P \cap A}$ follows.

Next we claim that $R_{P \cap R} =
D^{-1}(K_{P}) \cap S_{P}$.  Let $r \in R$ and $b \in R \setminus P$.   Then since $R = D^{-1}(K) \cap S$, we have  $D(r/b) = (bD(r) - rD(b))/b^2 \in K_{P}$, so that $D(R_{P \cap R}) \subseteq K_{P}$.  Thus $R_{P\cap R} \subseteq D^{-1}(K_{P}) \cap S_{P}$.  To see that the reverse inclusion holds, suppose that
 $x \in S_{P}$  such that $D(x) \in K_{P}$.  By our above argument,
$S_{P} = S_{P \cap A}$, so
 there exist $s \in S$ and $c \in A \setminus (P \cap A)$ such that $x = \frac{s}{c}$.  By assumption $D(\frac{s}{c}) \in K_{P}$, and since $D$ is $A$-linear, we have $\frac{1}{c}D(s) = D(\frac{s}{c}) \in K_{P}$.  Thus,
since $c \not \in P$, we conclude $D(s) \in K$.  Since $R = S \cap D^{-1}(K)$, this implies that $s \in R$, and hence $x  = \frac{s}{c} \in R_{P \cap R}$.  This proves the claim  that $R_{P \cap R} =
D^{-1}(K_{P}) \cap S_{P}$.

Finally we claim that $R_{P \cap R}$ is strongly twisted by $K_{P}$.  Indeed, we have verified that $R_{P \cap R} = D^{-1}(K_{P}) \cap S_{P}$.  Also, since $R$ is strongly twisted by $K$, $D(F)$ generates $FK$ as an $F$-vector space.  Moreover, $S \subseteq \Ker D + sS$ for all $0 \ne s \in S$, and as noted above $S_{P} = S_{P \cap A}$, so since $A_{P\cap A} \subseteq \Ker D$ we have that $S_{P} \subseteq \Ker D + sS_{P}$ for all $0 \ne s \in S$.  Thus $R_{P \cap R}$ is strongly twisted by $K_{P}$.
\end{proof}

\begin{prop} \label{global stable}
Suppose that $S$ is a Dedekind domain with quotient field $F$ having countably many maximal ideals, and that $S$ has a subring that is strongly twisted by a torsion-free $S$-module of infinite rank.     If
 $\{e_n\}_{n=1}^\infty$ is a sequence for which  each $e_n \in  {\mathbb{N}} \cup \{\infty\}$,
 then
   there exists a subring $R$ of $S$ having countably many maximal ideals $M_1,M_2,\ldots$ such that:
\begin{itemize}
\item[{(1)}]
 $R$ is a   stable domain  having normalization $S$ and quotient field $F$.

\item[{(2)}]
For each $n>0$,  $M_n$ is  minimally generated by $e_n+1$ elements.

\item[{(3)}] If  each $e_n$ is finite, then $R$ is a Noetherian domain.

\end{itemize}
\end{prop}

\begin{proof}
List the  maximal ideals of $S$ as  $N_1,N_2,\ldots$, and for each
$t \geq 1$, define $K_t = \bigoplus_{i=1}^{e_t} S_{N_t}$. Then
define $K = \bigoplus_{t=1}^\infty K_t$.  By Lemma~\ref{more}, $K$ is a strongly twisting module for $S$.
%
Let $D:F \rightarrow FK$ be the corresponding derivation that twists $R := S \cap D^{-1}(K)$.  Let $A =  S \cap \Ker D$.
Then for each $0 \ne a \in A$, since
$a$ is contained in at most finitely many of the $N_i$'s, there
exist positive integers $t_1,t_2,\ldots,t_m$ such that  $K/aK
\cong K_{t_1}/aK_{t_1} \oplus \cdots \oplus K_{t_m}/aK_{t_m}$ as
$S$-modules.  For each $t_i$, since $S_{N_{t_i}}$ is a DVR, it follows that
$S_{N_{t_i}}/aS_{N_{t_i}}$ is a cyclic $S$-module.  (This is because for each maximal ideal $N$ of $S$ and $k>0$, $S_N = S + N^kS_N$, so that since $S_N$ is a DVR, $S_N = S + aS_N$ for each $0 \ne a \in S$.)
  Thus if $e_{n}$ is finite for all
$n>0$, then
 $K/aK$ is a finite $S$-module.
Therefore, in this case by Theorem~\ref{Twisted Noetherian subrings}, $R$ is a
Noetherian domain, proving (3).   Also, regardless of whether all the $e_n$'s are
finite,  Theorem~\ref{fs} implies that $R$ is a  stable domain  with quotient field $F$ and normalization $S$, and this proves (1).

To prove (2),
for each $n$, let $M_n = N_n \cap R$. Then
each $M_n$ is a maximal ideal of $R$, and since $R \subseteq S$ is integral, every maximal ideal of $R$ is accounted for in this way. Fix $n$, and to simplify notation, let $M = M_n$ and $N = N_n$.   By Lemma~\ref{construction}, $R_{M}$ is a subring of $S_{N}$ that is  strongly twisted by $K_{N}$.  If the maximal ideal of $R_M$ is finitely generated, so that $R_M$ is a Noetherian domain, then since by Theorem~\ref{pre-construction completion},
 $(R_M)^{\widehat{\:}} \cong (S_N)^{\widehat{\:}} \star (K_N)^{\widehat{\:}}$ and $S_N$ is a DVR, the embedding dimension of $R_M$ is given by the following calculation (recall our notation $M = M_n$ and $N = N_n$):
\begin{eqnarray*}
 \embdim {R_{M}} & = &  1 + \dim_{S_{N}/NS_{N}} K_{N}/NK_{N} \\
 \: & = & 1 + \dim_{S_N/NS_N} K_n/NK_n \: \: = \: \: 1 + e_n.
 \end{eqnarray*}
 Thus for each $n$, either $R_{M_n}$ is a non-Noetherian ring or $R_{M_n}$ is Noetherian and its maximal ideal can be generated by $e_n+1$ but no fewer elements.
  Since every nonzero ideal of $R$ is contained in at most finitely many maximal ideals of $R$, an ideal of $R$ can be generated by $k$ elements, with $k\geq 2$, if it can  can be locally generated by $k$ elements \cite[Theorem 26, p.~35]{M2}.
Therefore, if $M_n$ is finitely generated, it  can be minimally generated by   $e_n+1$ elements.  This proves (2).
\end{proof}

\begin{cor}  \label{global stable example}
Assume that: \index{twisted subring!of characteristic $p$}
\begin{itemize}
\item[{(a)}]
  $k$ is a countable field of prime characteristic  that is a separably generated extension of infinite transcendence degree over a subfield, and

\item[{(b)}]  $\{e_n\}_{n=1}^\infty$ is a sequence for which  each $e_n \in  {\mathbb{N}} \cup \{\infty\}$.

\end{itemize}
Then there exists a subring $R$ of $k[X]$ having quotient field $k(X)$ such that $R$ is a stable domain with normalization $k[X]$ and the set of maximal ideals of $R$ can be written $\{M_1,M_2,\ldots\}$, where for each $n$, $M_n$ is minimally generated by $e_n+1$ elements.
\end{cor}

\begin{proof}  Since $k$ is countable, $k[X]$ is a PID having countably many maximal ideals.  Therefore, we may apply Theorem~\ref{existence strongly twisted cor} and
Proposition~\ref{global stable} to obtain a stable
subring $R$ of $S$ whose maximal ideals behave accordingly.
\end{proof}

 The proposition and its corollary  concern one-dimensional non-local twisted subrings.   The one-dimensional local case is treated extensively in \cite{OlbAR}, while more on local stable domains can be found in \cite{OlbGFF}.


\begin{thebibliography}{FHP34}


\bibitem{AbhLocal} S.~Abhyankar,
Local rings of high embedding dimension.
Amer. J. Math. 89 (1967) 1073--1077.


\bibitem{AW}  D.~D.~Anderson and M.~Winders,
Idealization of a module.
J. Commut. Algebra 1 (2009), no. 1, 3--56.










\bibitem{BH} W.~Bruns and J.~Herzog, {\it Cohen-Macaulay rings}.
Cambridge Studies in Advanced Mathematics, 39. Cambridge University Press, Cambridge, 1993.








\bibitem{Eis} D.~Eisenbud, {\it Commutative algebra. With a view toward algebraic
geometry.} Graduate Texts in Mathematics, 150. Springer-Verlag, New
York, 1995.





\bibitem{FR} D.~Ferrand and M.~Raynaud, Fibres formelles d'un anneau
local noeth\'erien.  Ann. Sci. \'Ecole Norm. Sup. (4) 3 1970 295--311.










\bibitem{GL} K.~R.~Goodearl and T.~H.~Lenagan,  Constructing bad Noetherian local
domains using derivations. J. Algebra 123 (1989), no. 2, 478--495.




\bibitem{HRS} W.~Heinzer, C.~Rotthaus and J.~Sally,
 Formal fibers
and birational extensions.  Nagoya Math. J.  131  (1993), 1--38.


\bibitem{HRW} W.~Heinzer, C.~Rotthaus and S.~Wiegand,
 Extensions of local domains with trivial generic fiber.  Illinois J. Math.
 51  (2007),  no. 1, 123--136.



\bibitem{Kap} I.~Kaplansky, {\it Set theory and metric spaces.} Second edition. Chelsea Publishing Co., New York, 1977.






\bibitem{Kunz} E.~Kunz, {\it K\"ahler differentials}. Advanced Lectures in
Mathematics. Friedr. Vieweg \& Sohn, Braunschweig, 1986.







\bibitem{Lipman} J.~Lipman, Stable ideals and Arf rings.  Amer. J. Math.  {\bf 93}  (1971), 649--685.










\bibitem{M2} E.~Matlis, {\it Torsion-free modules}, The University of Chicago
Press, 1972.


\bibitem{Ma}  H.~Matsumura, {\it Commutative Ring Theory}, Cambridge
University Press, 1986.

\bibitem{Mat2} H.~Matsumura, On the dimension of formal fibres of a
local ring, in: {\it Algebraic geometry and commutative algebra in
honor of Masayoshi Nagata},  Vol. I,  261--266, Kinokuniya, Tokyo,
1988.

\bibitem{Na}  M.~Nagata, {\it Local rings}.
Interscience Tracts in Pure and Applied Mathematics, No. 13, John
Wiley \& Sons, New York-London, 1962.

\bibitem{NewN} M.~Nagata, {\it Theory of commutative fields.}
 Translations of Mathematical Monographs, 125. American Mathematical Society, Providence, RI, 1993. 

\bibitem{OlStructure} B.~Olberding, On the structure of stable
domains, Comm. Algebra {30} (2002), no. 2, 877--895.

\bibitem{OlClass} B.~Olberding, On the classification of stable
domains, J. Algebra 243 (2001), 177--197.




\bibitem{OlbAR} B.~Olberding, One-dimensional bad Noetherian domains, submitted.   

\bibitem{OlbGFF} B.~Olberding, Generic formal fibers and analytically ramified stable rings, submitted.

\bibitem{OlbSub} B.~Olberding, Invariants of contracted ideals of subintegral extensions of Noetherian rings, in preparation.  




\bibitem{Reiten} I.~Reiten,
The converse to a theorem of Sharp on Gorenstein modules.
Proc. Amer. Math. Soc. 32 (1972), 417--420.


\bibitem{Rush1} D.~Rush, Two-generated ideals and representations of
abelian groups over valuation rings, J. Algebra 177 (1995), 77-101.






\bibitem{SV} J.~Sally and W.~Vasconcelos, Stable rings, J. Pure
Appl. Algebra 4 (1974), 319-336.


\bibitem{Swan}  R.~G.~Swan, On seminormality, J. Algebra 67 (1980), pp. 210--229.





\bibitem{Val} E.~Valtonen,  Some homological properties of commutative semitrivial ring extensions.  Manuscripta Math.  63  (1989),  no. 1, 45--68.



\bibitem{Wei}  C.~Weibel, $K$-theory and analytic isomorphisms,
 Invent. Math.  61  (1980), no. 2, 177--197.


\bibitem{ZS} O.~Zariski and P.~Samuel, {\it Commutative Algebra Vol. 2},
Van Nostrand, Princeton, 1958.









\end{thebibliography}
\end{document}